\documentclass[10pt]{article}

\usepackage{amsmath}
\usepackage{subcaption}
\usepackage{tikz}
\usepackage{tikz-cd}
\usetikzlibrary{shapes,arrows}
\tikzstyle{line}=[draw] 
\tikzstyle{cube}=[draw] 
\usepackage{amssymb}
\usepackage{amsthm}
\usepackage{graphicx}
\usepackage{mathtools}
\usepackage{caption}
\usepackage{pst-node}
\usepackage{wrapfig}
\usepackage{pgfplots,lineno}
\pgfplotsset{compat=1.7}

\newcommand{\cI}{\mathcal I}
\newcommand{\cC}{\mathcal C}
\newcommand{\cR}{\mathcal R}
\newlength{\bw}
\newtheorem{theorem}{Theorem}
\newtheorem{lemma}[theorem]{Lemma}

\newtheorem{problem}{Problem}

\newcommand{\f}[1]{\mathbf{#1}}

\usetikzlibrary{shapes.misc}

\tikzset{cross/.style={cross out, draw=black, minimum size=2*(#1-\pgflinewidth), inner sep=0pt, outer sep=0pt},cross/.default={1pt}}

\theoremstyle{definition}

\newtheorem{example}{Example}
\newtheorem*{remark}{Remark}
\begin{document}

\title{IGA using Offset-based Overlapping Domain Parameterizations}

\author{Somayeh Kargaran\footnote{Doctoral Program Computational Mathematics, Johannes Kepler University Linz, Altenberger Stra\ss e 69, 4040 Linz, Austria \& Software Competence Center Hagenberg, GmbH (SCCH), Hagenberg, Austria}, 
Bert J\"{u}ttler\footnote{Institute of Applied Geometry, Johannes Kepler University Linz, Altenberger Stra\ss e 69, 4040 Linz, Austria}, 
Thomas Takacs\footnote{Johann Radon Institute for Computational and Applied Mathematics (RICAM), Austrian Academy of Sciences, Altenberger Stra\ss e 69, 4040 Linz, Austria \& Institute of Applied Geometry, Johannes Kepler University Linz, Austria}}

\date{\today\footnote{This manuscript is published as: Kargaran, J\"uttler, Takacs. IGA Using Offset-based Overlapping Domain Parameterizations. \emph{Computer-Aided Design}, 139: 103087, 2021. The paper is available online at: https://doi.org/10.1016/j.cad.2021.103087}}
 
\maketitle

\section*{Abstract}

  Isogeometric analysis (IGA) is a numerical method, proposed in~\cite{IGA1}, that
connects computer-aided design (CAD) with finite element
analysis (FEA).
In CAD the computational domain is usually represented by B-spline or
NURBS patches. Given a B-spline or NURBS parameterization of the domain, an isogeometric discretization is defined on the domain using the same B-spline or NURBS basis as for the domain parameterization. Ideally, such an isogeometric discretization allows an exact representation of the underlying CAD model.

CAD models usually represent only the boundary of the object. For planar domains, the CAD model is given as a collection of curves representing the boundary. Finding a suitable parameterization of the interior is one of the major
issues for IGA, similar to the mesh generation process in the FEA setting. 
The objective of this isogeometric parameterization problem is to obtain a set of patches, which exactly represent the boundary of the domain and which are parameterized regularly and without self-intersections. This can be achieved by segmenting the domain into patches which are matching along interfaces, or by covering  the domain with overlapping patches. In this paper we follow the second approach. 

To construct from a given boundary curve a planar parameterization suitable for IGA, we propose an offset-based domain parameterization algorithm.
Given a boundary curve, we obtain an inner curve by generalized offsetting. 
The inner curve, together with the boundary curve, naturally defines a ring-shaped patch with  an associated parameterization.
By definition, the ring-shaped patch has a hole, which can be covered by a multi-cell domain. Consequently, the domain is represented as a union of two overlapping subdomains which are regularly parameterized. On such a configuration, one can employ the overlapping multi-patch (OMP) method, as introduced in~\cite{KARGARAN}, to solve PDEs on the given domain.
The performance of the proposed method is reported in several numerical
examples, considering different shape properties of the given boundary curve.

\section{Introduction}

Isogeometric Analysis (IGA) is a computational approach, proposed in~\cite{IGA1} by Hughes et al., connecting computer-aided design (CAD) and finite element analysis (FEA). In the IGA framework the same basis functions are used for describing the geometry and for the numerical analysis. On a domain given by a B-spline parameterization, isogeometric test and trial functions are defined by composition of B-splines with the inverse of the domain parameterization. Thus it is possible to perform simulations directly on the geometry representation of CAD models.
An overview and summary of IGA can be found in~\cite{IGA3}.

One of the advantages of IGA -- when compared to the finite element method (FEM) -- is the ability to exactly represent computational domains from CAD using B-splines or NURBS. 
In FEM, one first has to obtain a discrete mesh from a given CAD model. This mesh generation process is, in general, expensive. In IGA, the geometry of the computational domain is often given directly from the CAD model. However, in a CAD model, the computational domain is usually given by a boundary representation, that is, by a (collection of) boundary curves in 2D or surfaces in 3D. Hence, obtaining a spline representation of the interior of a complex domain from a given CAD description of its boundary is a big challenge in IGA.

To obtain a parameterization of a domain given only by its boundary requires, in general, a segmentation of the domain as a first step. More precisely, if a collection of boundary curves is given, one needs to divide the interior into several segments, such that each segment can be parameterized using a simple patch. Hence, after the segmentation step, appropriate patch parameterization methods can be performed on each single segment. The parameterization of the entire computational domain is then given as a collection of patches, a so-called multi-patch parameterization.

In the following, we give an overview of approaches for segmentation as well as for parameterizations of computational domains. There exist several methods to segment the interior of the computational domain, e.g.,~\cite{Geo-Jinlan,LIU2015162} derived from the skeleton of the domain, \cite{BUCHEGGER20172} based on patch adjacency graphs, \cite{XU2018175} based on quad meshing or~\cite{Sajavicius2019,FALINI2019390} using template segmentations. To obtain multi-patch volume segmentations for IGA, one may apply the isogeometric segmentation pipeline~\cite{JUTTLER201474,NGUYEN2014426,10.1007/978-3-319-23315-4_3,NGUYEN201689,HABERLEITNER2017135,HABERLEITNER2019179} or~\cite{Geo-Bernhard,Geo-Gang1}. 

The accuracy of a numerical simulation method performed on the computational domain depends on the quality of the patch parameterization. Therefore, we need to apply a parameterization method, in which the resulting patches are regular, smooth and without any self-intersections. The parameterization may be obtained by optimizing some functional measuring its quality, see e.g.~\cite{Geo-Gravesen,Geo-Gang,Sajavicius2019}. Alternatively, one may obtain a patch parameterization as the solution of a suitable PDE, such as in~\cite{hinz2018iga}, in~\cite{Geo-MARTIN,Geo-Gang2} using harmonic functions, or based on a quasi-conformal Teichm\"uller map as in~\cite{Geo-Nian}. The methods in~\cite{SU201742,ZHENG201928} are developed using optimal mass transport. 
Other patch parameterization techniques include, e.g., low-rank parameterizations as in~\cite{PAN20181}, swept volume constructions as in~\cite{Geo-Juettler} or parameterizations based on offsetting as in~\cite{nguyen2014isogeometric}. The parameterization may also be derived from a polysquare or polycube representation as in~\cite{LIU2015162,yu2014optimizing,XIAO201829}, which can be seen as a pixelized or voxelized approximation of the domain.

The parameterization method depends on the requirements on the resulting domain parameterization. For instance, one usually does not allow a multi-patch parameterization where the patches overlap. Thus, most multi-patch parameterization methods result in segmentations where the patches are matching along interfaces, cf. Figure~\ref{fig:parameterization-strategies-c}. In CAD models overlapping patches may occur as a result of a Boolean union. In that case one of the patches is usually trimmed to remove the overlap. However, trimmed patches need additional care and are thus often reparameterized by a collection of regular, untrimmed patches, as in~\cite{MASSARVWI} and~\cite{HUI2005859}. Alternatively, one can apply methods that are able to handle overlapping subdomains, such as overlapping Schwarz methods, cf.~\cite{BercovierSoloveichik15_3756}, the overlapping multi-patch (OMP) method developed in~\cite{KARGARAN}, or the method introduced in~\cite{Antolin}, which employs Nitsche's method to couple overlapping patches. Note that while~\cite{MASSARVWI,HUI2005859} provide (re)parameterization strategies for domains obtained by trimming, the papers~\cite{BercovierSoloveichik15_3756,KARGARAN,Antolin} introduce solution methods for domains composed of overlapping patches without providing parameterization strategies.

In this paper we propose a new parameterization method for IGA, the so-called offset-based overlapping domain parameterization (OODP) method. The OODP method results in a parameterization formed by two overlapping patches, a ring-shaped patch with a hole and another patch covering the hole. The construction of the ring-shaped patch is based on generalized offsetting. Such a segmentation process is easy and can be performed for a considerable range of given boundary curves in 2D. Moreover, the approach can also be extended easily to 2.5D domains, which are constructed from extrusion or sweeping of a planar or surface domain, respectively.

\begin{figure}[ht]\centering
\begin{subfigure}{.15\textwidth}
 \includegraphics[width=0.99\textwidth]{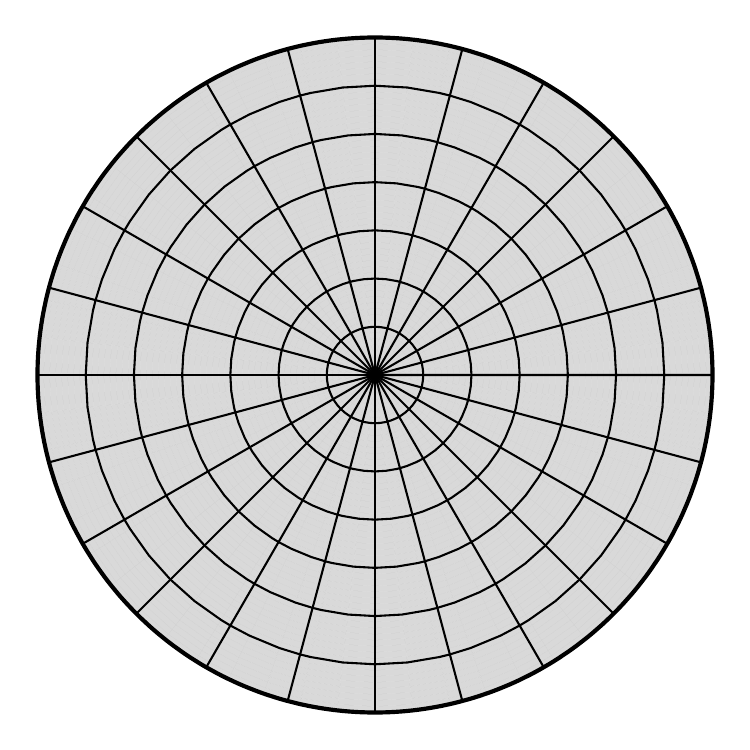}
 \subcaption{}\label{fig:parameterization-strategies-a}
\end{subfigure}
\begin{subfigure}{.15\textwidth}
 \includegraphics[width=0.99\textwidth]{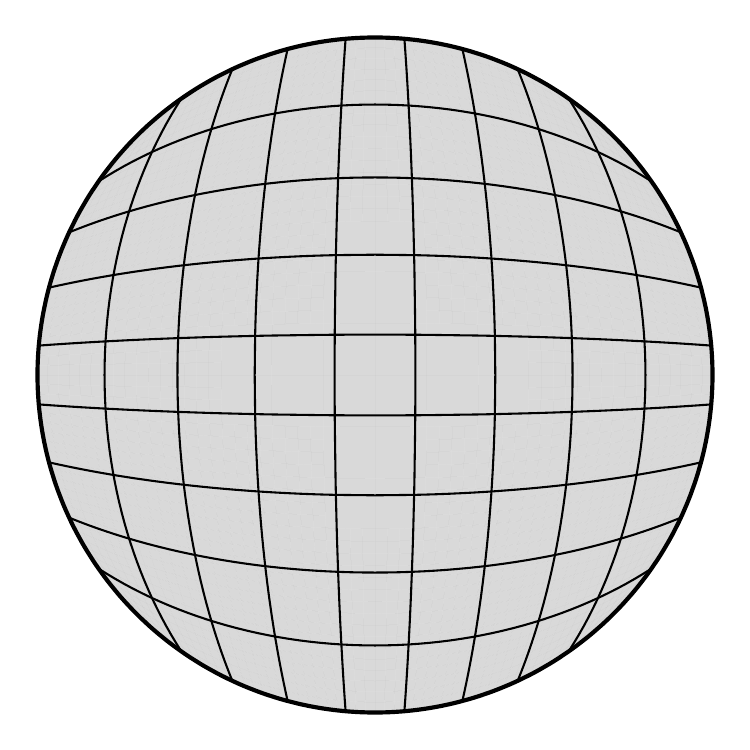}
 \subcaption{}\label{fig:parameterization-strategies-b}
\end{subfigure}
\begin{subfigure}{.15\textwidth}
 \includegraphics[width=0.99\textwidth]{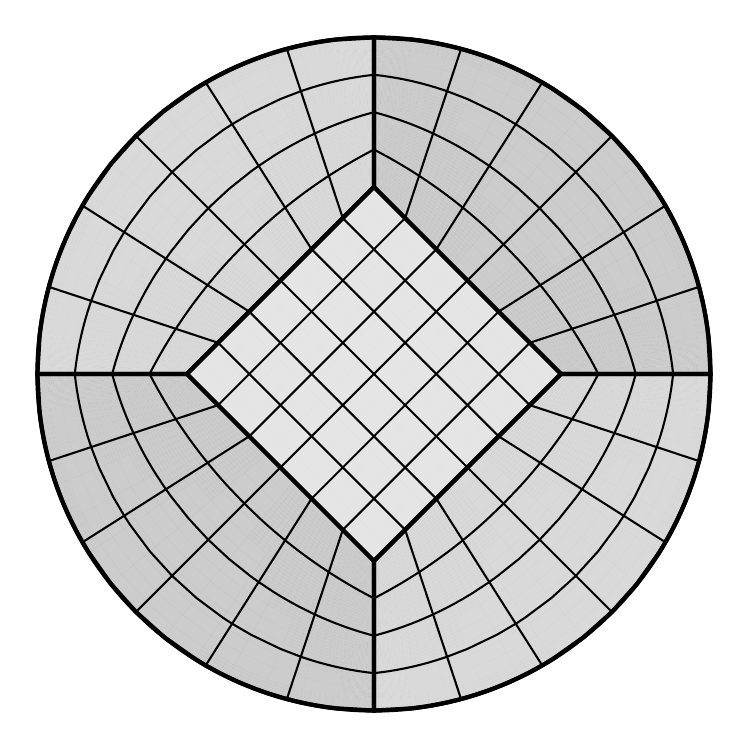}
 \subcaption{}\label{fig:parameterization-strategies-c}
\end{subfigure}
\begin{subfigure}{.15\textwidth}
 \includegraphics[width=0.99\textwidth]{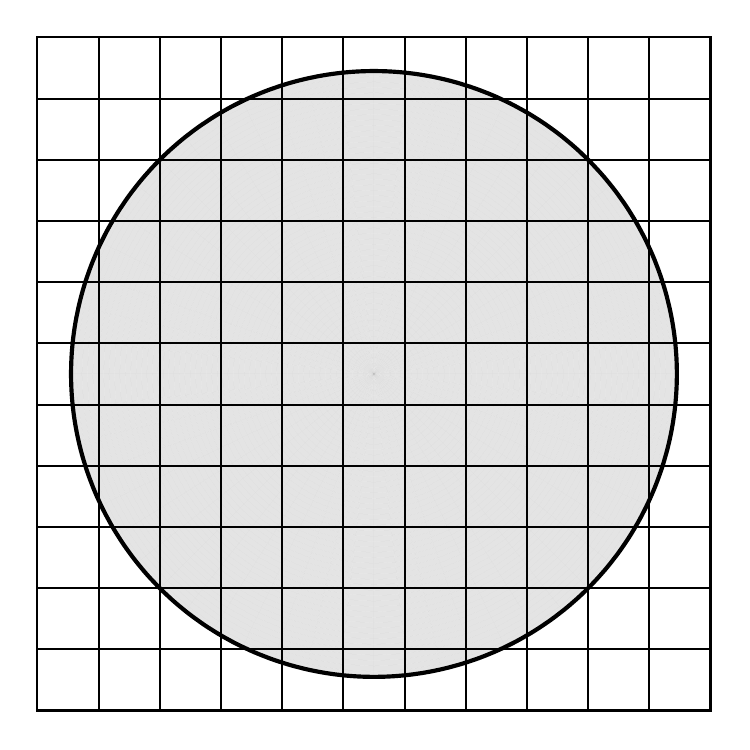}
 \subcaption{}\label{fig:parameterization-strategies-d}
\end{subfigure}
\begin{subfigure}{.15\textwidth}
 \includegraphics[width=0.99\textwidth]{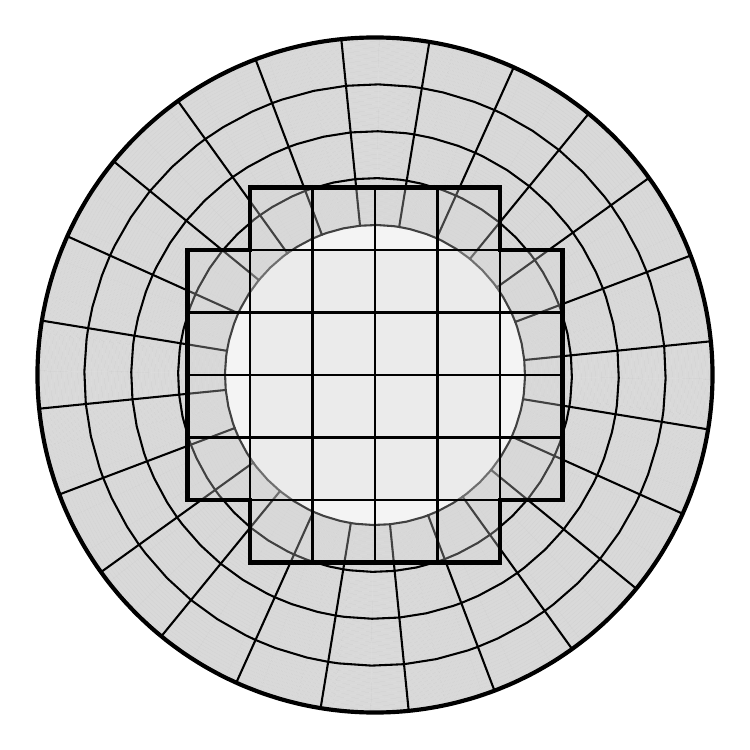}
 \subcaption{}\label{fig:parameterization-strategies-e}
\end{subfigure}
 \caption{Overview of different parameterization strategies. From left to right: polar, singular, $C^0$-matching multi-patch, trimmed/immersed boundary and offset-based overlapping domain parameterization.}
 \label{fig:parameterization-strategies}
\end{figure}

In Figure~\ref{fig:parameterization-strategies} we show examples of domain parameterizations following different strategies. The result of the OODP method is shown in Figure~\ref{fig:parameterization-strategies-e}. The domain is then represented as a union of two overlapping subdomains. Thus the OODP method results in a smaller number of subdomains than most segmentation based parameterization methods (e.g.~\cite{Geo-Jinlan,LIU2015162,BUCHEGGER20172,XU2018175,Sajavicius2019,FALINI2019390}). See Figure~\ref{fig:parameterization-strategies-c}, as an example where five patches are needed.

On the other hand, methods that parameterize a domain using a single patch, e.g.~\cite{Geo-Nian}, often contain singular points and result in strongly distorted meshes, cf. Figure~\ref{fig:parameterization-strategies-b}, where some of the mesh elements close to the boundary degenerate to triangles. Similarly, polar like parameterizations as in Figure~\ref{fig:parameterization-strategies-a} and the scaled boundary parameterization proposed in~\cite{ARIOLI2019576} result in a polar singularity in the interior that one needs to handle properly. 

The OODP method produces a ring-shaped patch which, for star-shaped domains, is similar to the scaled boundary parameterization proposed in~\cite{ARIOLI2019576}. The difference is that for the OODP method the polar singularity is removed and the resulting hole is covered with another patch. Thus, all domains that possess a scaled boundary parameterization also possess an overlapping parameterization, where both the ring-shaped patch and multi-cell domain are regularly parameterized.

Moreover, the OODP method can be seen as an alternative to the method proposed in~\cite{sanches}, based on an immersed boundary curve, cf. Figure~\ref{fig:parameterization-strategies-d}. 
There one considers a boundary curve immersed in a regular grid. The physical domain is then given by those (cut) grid cells in the interior of the given curve. The grid cells that cover the physical domain can thus be interpreted as a multi-cell domain with cut cells. Since the boundary curve cuts some of the cells, the basis functions need to be cut as well along the prescribed boundary curve. To obtain a stable discretization, some basis functions that have support near the boundary need to be modified.

In our method, instead of cutting the basis functions that have support at the boundary, we drop all functions that are close to the boundary. To be able to approximate any function on the physical domain, the region close to the boundary is covered by the ring-shaped patch. Since we can handle such overlapping patches directly, we do not need to apply any modification of the basis functions.

To summarize, the OODP method produces a domain parameterization composed of only two patches, which are regular and possess a simple structure. Since the method aims at minimizing the distortion of mesh elements and yields regular patches, the resulting discretization is directly suitable for isogeometric analysis, as shown in Section~\ref{Numerical-examples}. Due to the simple patch structure, the system matrices can be assembled efficiently.

The outline of this paper is given as follows. 
In Section~\ref{outline} we give an overview of the input and output of the OODP method. 
The OODP algorithm is described and its performance is studied experimentally in Section~\ref{sec-algorithm-1}. In Section~\ref{sec:curves-with-corners} we extend the method to boundary curves with corners. We employ the overlapping multi-patch (OMP) method, which is proposed in~\cite{KARGARAN} and briefly summarized in Section~\ref{sec:omp}, to solve second order PDEs on the resulting domain parameterizations. 
Finally, we provide numerical experiments in Section~\ref{Numerical-examples}.

\section{The offset-based overlapping domain parameterization method}\label{outline}

Given a simply-connected planar domain represented by its boundary curve, we propose an algorithm to generate a parameterization of the domain consisting of two overlapping patches. First we construct a ring-shaped patch from the boundary curve by generalized offsetting. The part of the domain which is not parameterized by this ring-shaped patch is then covered with a multi-cell domain. We call this approach \emph{offset-based overlapping domain parameterization} (OODP). We first discuss the structure of the input and output in Section~\ref{in-out-put}. Then, in Section~\ref{offset-algorithm-steps} we give a step-by-step overview of the OODP strategy. In Section~\ref{subsec:omp} we summarize how the overlapping multi-patch formulation developed in~\cite{KARGARAN} is applied to the resulting two-patch parameterization.
\subsection{Structure of input and output}\label{in-out-put}

As an input we consider a regularly parameterized, $1$-periodic, smooth, simple spline curve $  \mathbf{C}_{B}(t):  \mathbb{R}\rightarrow \mathbb{R}^2 $, with counter-clockwise orientation, representing the boundary of the domain $\Omega$. We assume $\mathbf{C}_{B} \in (\mathcal{S}^{p}_h)^2$, where $\mathcal{S}^{p}_h$ is a spline space of degree $p$ with $1$-periodic knot vector $\Xi_h$ over $\mathbb{R}$. 
From this boundary curve we construct a ring-shaped patch as shown in Figure~\ref{covering-hole} (left).
By definition, the ring-shaped patch has a hole in the middle, which we cover with a multi-cell domain.
A multi-cell domain consists of a finite set of edge-connected cells that are defined as follows
\begin{equation}\label{cells}
C_{ij} = [h_c i , h_c{(i+1)}] \times [h_c j , h_c{(j+1)}], \quad (i,j) \in \mathcal{I}_M,
\end{equation}
where $\mathcal{I}_M$ is a finite index set and $h_c>0$. The multi-cell domain $\Omega^{\cC}$ is then given as the interior of the union of all $C_{ij}$ for $(i,j) \in \mathcal{I}_M$.

Therefore, the output of our algorithm consists of a ring-shaped patch $\Omega^{\cR}$ with parameterization $\f F$, covering a neighborhood of the boundary curve, as well as a multi-cell domain $\Omega^{\cC}$ covering the hole $\Omega \setminus \Omega^{\cR}$, see Figure~\ref{covering-hole}.

\subsection{The OODP algorithm}\label{offset-algorithm-steps}

In the following we present the algorithm to construct an OODP for a smooth curve as given in Section~\ref{in-out-put}. We will consider curves with corners in more detail in Section~\ref{sec:curves-with-corners}. Therefore, we construct an open, ring-shaped patch $\Omega^{\cR}$ having the parameterization $\f{F} : \left]0,1\right[ \times \left[0,1\right[$, with 
\begin{equation*}\label{parameterization0}
\f{F} (s,t) = \f{C}_{B}(t) \cdot (1-s) + \f{C}_{I}(t) \cdot s, 
\end{equation*}
satisfying the periodicity condition $\f{F} (s,0) = \f{F} (s,1)$, where the value at $t=1$ has to be considered in the limit. By construction, the curves $\mathbf{C}_{B}$ and $\mathbf{C}_{I}$ are in the same $1$-periodic spline space $\mathbf{C}_{B},\mathbf{C}_{I} \in (\mathcal{S}^{p}_h)^2$. We want the parameterization to be regular, i.e.,
\[
\left|\det \nabla \f F (s,t)\right| \geq c > 0, \quad \mbox{for all }(s,t)\in \left]0,1\right[ \times \left[0,1\right[.
\]

The algorithm consists of the following steps.
\subsubsection*{Step I: Construct a generalized inner offset curve}

Given a boundary curve $ \f{C}_{B} $ we define a generalized inner offset curve $ \f{C}_{O} $ as follows
\begin{equation}\label{offset-curve}
\f{C}_{O}(t) = \f{C}_{B}(t) + \mu(t)\cdot \f{q}(t), \quad t \in \left[0,1\right[,
\end{equation}
where $\mu(t) > 0$ and $ \f{q} $ is a predefined quasi-normal vector to the curve $ \f{C}_{B} $, that is, a vector which is continuous in $t$, $1$-periodic, non-tangential and pointing inwards (a generalization of a normal vector). The function $\mu$ is assumed to be a $1$-periodic spline of degree $p$, which is determined by a regularized, quadratic minimization problem, which will be described later.

Note that the curve $ \f{C}_{O} $ is a generalization of a classical offset curve. Indeed, if $\f q$ is the unit normal to $ \f{C}_{B} $ and $\mu$ is chosen to be a constant, then $ \f{C}_{O} $ is a classical offset curve.
By construction, the curve $ \f{C}_{O} $ is regularly parameterized, $1$-periodic and simple.

\subsubsection*{Step II: Fit a B-spline parameterization to the ring-shaped patch}

Given $ \f{C}_{B} $ and $ \f{C}_{O} $, as constructed in Step~I, we define
\begin{equation}\label{parameterization}
\widetilde{\f{F}}(s,t) = (1-s) \f{C}_{B}(t) + s\,\f{C}_{O}(t) = \f{C}_{B}(t) + s\,\mu(t) \f{q}(t), 
\end{equation} 
where $ (s,t) \in \left]0,1\right[ \times \left[0,1\right[$.
Since $\widetilde{\f{F}}$ is, in general, not a spline parameterization, we solve the following fitting problem
\begin{eqnarray*}
&\min_{\f F \in (\f S_1^1 \otimes \f S_h^p)^2} \| \f F - \widetilde{\f{F}}\|_{L^2(\left]0,1\right[^2)}^2, &\label{fitting} \\
&\mbox{such that }\;\f F(0,t) = \widetilde{\f{F}}(0,t)\;\mbox{ and }\;\f F(s,0) = \f{F}(s,1), & \nonumber
\end{eqnarray*}
where $\f S_1^1$ is the space of polynomials of degree $1$ over $\left]0,1\right[$ and $\f S_h^p$ is the spline space of degree $p$ and mesh size $h$ containing the input curve.

Under certain conditions, as specified in Theorem~\ref{mu-max-theorem} in Section~\ref{sec-algorithm-1}, the parameterization $\widetilde{\f{F}}$ is regular. Assuming moreover that the fitting error $\| \f F - \widetilde{\f{F}}\|_{L^2}$ is sufficiently small, then the B-spline parameterization $\f{F}$ is also regular. 
Thereby, we set the ring-shaped patch to be $\Omega^{\cR} = \f F(\left]0,1\right[ \times \left[0,1\right[)$.

A short computation confirms that the minimization problem simplifies to
\begin{equation*}
\min_{\begin{array}{c}\f{C}_{I} \in (\f S_h^p)^2 \\ \f{C}_{I}(0)=\f{C}_{I}(1)\end{array}} \| \f{C}_{I}(t) - \f{C}_{O} (t)\|_{L^2(\left]0,1\right[)}^2.
\end{equation*}
\subsubsection*{Step III: Cover the hole by a multi-cell domain}

From the previous step we obtain a ring-shaped patch ${\Omega^{\cR} \subset\overline\Omega}$. We cover the hole $\Omega \setminus \Omega^{\cR}$ inside the ring, using a multi-cell domain $\Omega^{\cC}$, such that $\Omega \setminus \Omega^{\cR} \subset \Omega^{\cC} \subset \Omega$.
An example of a ring-shaped patch is illustrated in Figure~\ref{covering-hole} (left).
We cover the hole by a multi-cell domain, see Figure~\ref{covering-hole} (center), such that the patches are overlapping. 
The resulting parameterization is shown in  Figure~\ref{covering-hole} (right).
\setlength{\bw}{0.6 cm}
\begin{figure}[ht]\centering
	\footnotesize 
	\includegraphics[width=20\bw]{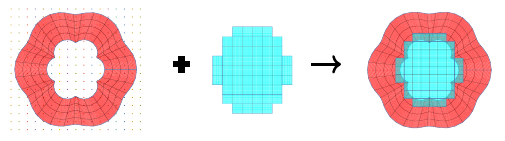}
	\caption{A ring-shaped domain $\Omega^{\cR}$ (red) and a multi-cell domain $\Omega^{\cC}$ (blue) covering the hole $\Omega \setminus \Omega^{\cR}$.}\label{covering-hole}\label{flower-grid}
\end{figure}

To explain the algorithm for creating a multi-cell domain, we consider Figure~\ref{flower-grid} (left), which depicts a ring-shaped 
patch $\Omega^\cR$ and a lattice with distance $h_c$ in both directions. This lattice creates a set of cells $C_{ij}$ as in~\eqref{cells}. All cells that satisfy 
\begin{equation*}
C_{ij} \cap ( \Omega \setminus \Omega^\cR ) \neq \emptyset
\end{equation*}
 are selected for the multi-cell domain $\overline{\Omega^\cC} = \bigcup_{i,j \in \mathcal{I}_M} C_{ij}$. 
If the lattice distance $h_c$ is sufficiently small, the obtained multi-cell domain $\Omega^\cC$ and the ring-shaped domain are overlapping and, moreover, the cells are edge-connected and there are no isolated cells. Moreover, if $h_c$ is chosen small enough, then $\Omega^\cC \subset \Omega$.
Thereby, we end up with a domain parameterization, which is a union of two overlapping subdomains (see, Figure~\ref{flower-grid} (right)).
\subsection{The overlapping multi-patch formulation}\label{subsec:omp}

From the algorithm described above, we obtain open, overlapping patches $\Omega^{\cR}$ and $\Omega^{\cC}$ covering the open domain $\Omega = \Omega^{\cR} \cup \Omega^{\cC}$. We solve a PDE on the domain $\Omega$ using the overlapping multi-patch (OMP) method as proposed in~\cite{KARGARAN}.

Assume that $\mathcal{A}(.)$ and $\mathcal{ L}(.)$ are appropriate multilinear forms derived from a PDE. The continuous problem looks as follows:
find $(u^{\cR}_0,u^{\cC}_0) \in H_0^1(\Omega^{\cR}) \times H_0^1(\Omega^{\cC}) $ and $(u^{\cR}_M,u^{\cC}_M) \in H^1(\Omega^{\cR}) \times H^1(\Omega^{\cC}) $ such that
\begin{eqnarray}\centering 
\mathcal{A}((u_0^{\cR} + u^{\cR}_M,u_0^{\cC} + u^{\cC}_M),(v^{\cR}_0,v^{\cC}_0)) &=&  \mathcal{ L}(v^{\cR}_0,v^{\cC}_0), \quad\label{variationaleq0} \\
C(u_0^{\cR} , u_0^{\cC} , u_M^{\cR} , u_M^{\cC} ) &=& \f 0 \label{coupling-cond},
\end{eqnarray}
for all $ (v^{\cR}_0,v^{\cC}_0) \in H_0^1(\Omega^{\cR}) \times H_0^1(\Omega^{\cC}) $.
Here, \eqref{variationaleq0} corresponds to a variational formulation of a PDE and~\eqref{coupling-cond} is a coupling condition on the coupling boundaries $\partial\Omega^{\cR} \cap \Omega$ and $\partial\Omega^{\cC} \cap \Omega$. The formulation is derived in more detail in Section~\ref{sec:omp}.

For discretizing the variational equation~\eqref{variationaleq0}, we use a standard Galerkin approach. The coupling condition~\eqref{coupling-cond} is discretized using a collocation scheme at the coupling boundaries.
Since the patch $\Omega^{\cR}$ has a $1$-periodic parameterization, we use isogeometric basis functions based on standard periodic B-splines to discretize the functions $u^{\cR}_0$, $v^{\cR}_0$ and $u^{\cR}_M$ on $\Omega^{\cR}$.
The functions $u^{\cC}_0$, $v^{\cC}_0$ and $u^{\cC}_M$ can be discretized using standard B-splines.

\section{Constructing a generalized inner offset curve}\label{sec-algorithm-1}

In this section we provide details on Step~I of the OODP algorithm, constructing a generalized inner offset curve. In Section~\ref{Algorithm-1} we present an algorithm to construct the curve, as given in~\eqref{offset-curve}, from a given boundary curve and quasi-normal vector along the boundary.
In Section~\ref{Algorithm-1-example} we study the performance of the algorithm on several examples. Moreover, we consider the dependence of the algorithm on its parameters. In Section~\ref{comp-q(t)} we discuss the construction of suitable quasi-normal vectors.

\subsection{The generalized offsetting algorithm}\label{Algorithm-1}

In the following we propose an algorithm to find, for any given, $1$-periodic boundary curve $ \f{C}_B $ and quasi-normal vector $\f{q}$, a generalized inner offset curve $ \f{C}_{O} $, such that the resulting ring-shaped patch $\tilde{\f F}$ is regular and smooth. After providing an overview of the algorithm and its parameters, we discuss the influence of those parameters on the resulting patch. If the parameters are selected properly, the patch parameterization $\tilde{\f F}$ is regular.

The generalized offsetting algorithm needs as input a 
\begin{itemize}
	\item $1$-periodic boundary curve $\f{C}_{B}$,
	\item $1$-periodic quasi-normal vector $\f{q}$ to $\f{C}_{B}$,
	\item offsetting parameters $0<c<1$, $0<d$, and
	\item regularization parameters $\alpha\geq 0$ and $\beta\geq 0$.
\end{itemize}
Before we introduce the algorithm, we need to specify when the resulting offset curve is suitable for constructing a ring-shaped patch: we call an inner offset curve $\f{C}_{O}$ \emph{valid}, if it is $1$-periodic and the resulting patch
\[
 \widetilde{\f{F}}(s,t) = (1-s) \f{C}_{B}(t) + s\,\f{C}_{O}(t)
\]
is regularly parameterized, completely inside the boundary curve and possesses no foldovers. In that case we also refer to the patch as \emph{valid}.

The steps of the algorithm are given by: 
\begin{enumerate} 
	\item Compute 
	\begin{equation}\label{eq:def-mumax}
	\mu_{\max}(t) = \left\{
	\begin{array}{ll}
	\frac{\det(\f{C}_B'(t),\f q(t)) }{\det(\f{q}(t) ,\f{q'}(t))} & \mbox{ if }  0 < \det(\f{q}(t) ,\f{q'}(t)) \\
	\infty & \mbox{ otherwise,}
	\end{array}\right.
	\end{equation}
	for all $t \in [0,1[ $. 
	\item Set $ \mu_{\mathrm{target}}(t) = \min \{ c \cdot \mu_{\max}(t),d\}$.
	\item Find $\mu \in \mathcal{S}^{p}_h$ minimizing the quadratic energy functional
	\begin{equation*}\label{min-func}
	\|\mu - \mu_{\mathrm{target}}\|_{L^2([0,1])}^2 + \alpha \|\f{C}_{O}^{(1)} \|_{L^2([0,1])}^2 + \beta\|\f{C}_{O}^{(2)}\|_{L^2([0,1])}^2\rightarrow \min
	\end{equation*}
	where $\f{C}_{O}^{(i)}$ denotes the $i$-th derivative of 
	\begin{equation*}
	 \f{C}_{O} = \f C_B + \mu \f q .
	 \end{equation*}
	\item If the resulting generalized offset curve $\f{C}_{O}$ is valid, the algorithm terminates. Otherwise, we alternately shrink one of the parameters $d$, $h$, $\alpha$ or $\beta$ and repeat steps~2 to~4.
\end{enumerate}
Step~3 is a (regularized) quadratic minimization problem, which results in a linear system to be solved. Even though the system may need to be solved several times, its system matrix has to be assembled only once. The question remains, when the termination criterion in step~4 is satisfied. A necessary condition is that the resulting patch is regular. The regularity of the patch depends only on the boundedness of $\mu$ as summarized in the following theorem.

\begin{theorem}\label{mu-max-theorem}
	Assume that the patch parameterization~$\tilde{\f{F}}$ is given as in~\eqref{parameterization}
	where $\f{C}_B$ is a $1$-periodic, counter-clockwise oriented, simple curve and $\f q$ is a corresponding quasi-normal vector. 
	Moreover, assume that $ \mu$ is a $1$-periodic and continuous function. 
	If for all $t$ we have $0 < \mu(t) < \mu_{\max}(t)$, where $\mu_{\max}$ is defined as in~\eqref{eq:def-mumax}, then the parameterization $\tilde{\f{F}}$ is regular.
\end{theorem}
\begin{proof}
	See \ref{prooftheorem1}.
\end{proof} 
Assuming that the ring-shaped patch is regular, the offset curve $\f{C}_O$ is valid if it is simple (without self-intersections), oriented counter-clockwise and inside the domain. This yields the following result.
\begin{lemma}
	Let $\alpha=\beta=0$. Then there exists an $h>0$ and a $d>0$, both sufficiently small, such that $\f{C}_O$ is valid and the algorithm terminates.
\end{lemma}
\begin{proof}
	For $d>0$ small enough the curve $\f C_B + \mu_{\mathrm{target}} \f q$ is simple, oriented counter-clockwise and inside the domain. Consequently, if $h$ is small enough, the approximation $\mu$ minimizing $\|\mu - \mu_{\mathrm{target}}\|_{L^2([0,1])}^2$ yields a valid inner offset curve $\f{C}_O = \f C_B + \mu \f q$, as described above.

\end{proof} 
However, for fixed, non-zero regularization parameters $\alpha$ and $\beta$ the algorithm is not guaranteed to terminate. Consequently, the regularization parameters are shrunk in step~4.

The width of the ring-shaped patch (the distance between the curves $\f C_B$ and $\f C_O$) depends on the choice of the parameters $c$ and $d$, larger constants yielding a wider patch. However, too large constants may result in a patch which is not valid. Hence, the goal is to create a valid ring-shaped patch which is sufficiently wide, at least as wide as the desired mesh resolution for the numerical solution.

\begin{remark}
In the following we discuss the choice of parameters in the generalized offsetting algorithm:
\begin{itemize}
	\item If $c$ is close to zero, the ring-shaped patch will be too narrow. On the other hand, if $c$ is close to one, the patch parameterization is almost singular. Therefore, in practice, we set $c=\frac12$. 
    \item For each fixed $t$ there is an upper bound $d_{\max}(t)$, such that the ray
    \[
     \{ \f C_B(t) + s \, \f q(t) : 0 < s < d_{\max}(t) \}
    \]
    is completely inside the domain. Thus, the parameter $d$ should be chosen smaller than $d_{\max}(t)$ for any $t\in[0,1]$. Such an upper bound can be computed efficiently at sampled points and used as an initial guess. To avoid foldovers of the patch, i.e., self-intersections of the offset curve $\f C_O$, it may be necessary to select an even smaller value for $d$. Note that the constant $d$ in step~2 may also be replaced by a $1$-periodic function depending on the parameter $t$.

    \item Thus, in step~4 of the algorithm we shrink $d$ by setting $d' = \lambda \cdot d$, with $\lambda \in \left]0,1\right[$ fixed, and decrease the value of the regularization parameters by a factor of $10$. We moreover assume that $h$ is already small enough.
	\item The shape of the resulting ring-shaped patch depends strongly on the regularization parameters $ \alpha $ and $ \beta $ and the choice of the constants $c$ and $d$.
	If $c$ and $d$ are small enough and $\alpha = \beta = 0$, then the parameterization is regular.
	However, non-zero values of $\alpha$ and $\beta$ often give better results.
	\item In step 3 of the algorithm we assume $\mu \in \mathcal{S}^{p}_h$. The choice of the underlying spline space has an effect on the shape and smoothness of the resulting ring-shaped patch. Hence, selecting a higher degree and/or spline space with smaller mesh size $ h $, may improve the results.
\end{itemize}
\end{remark}

In the following section we consider several examples to show how we can create suitable parameterizations using the generalized offsetting algorithm. The goal is to find a quasi-normal vector and regularization parameters such that the resulting ring-shaped patch covers a large part of $\Omega$.

\subsection{A study of the dependence of the generalized offsetting algorithm  on its parameters}\label{Algorithm-1-example}

We use the generalized offsetting algorithm to obtain suitable, valid ring-shaped parameterizations for several given boundary curves.
In Example~\ref{example-with-normal} the function $\f q$ is chosen to be the normal vector. This does not always result in a satisfactory parameterization, as can be seen in Example~\ref{example-with-normal-2}. Hence, a suitable quasi-normal vector to the outer curve needs to be defined. 
In Example~\ref{star-shape-curve} we consider a star-shaped domain and we discuss the importance of the choice of $\f q$ to obtain a regular ring-shaped patch of sufficient width.
In Section~\ref{comp-q(t)}, we discuss how to improve the choice of $\f q$ locally for a given boundary curve. This is considered
in Example~\ref{remark-example}. In all examples we consider either an unregularized minimization problem in step~3 of the algorithm ($\alpha=\beta=0$) or we penalize only first derivatives ($\alpha>0$, $\beta=0$) or only second derivatives ($\alpha=0$, $\beta>0$).

\begin{remark}
	In the following examples, we assume that the function $\mu$ is a periodic B-spline of degree $p=3$. The size $h$ of the spline space for $\mu$ is set to $0.02$ for all the examples. 
\end{remark}

\begin{example}\label{example-with-normal} 
	We consider the peanut-shaped curve $\f C_B^1$, which is shown in Figure~\ref{peanut-result00}. In Figure~\ref{peanut-result101} we show the resulting patch parameterization where the function $\f q$ is the normal vector $\f n_B$ to the curve $\f C_B^1$ and $\alpha $ and $\beta $ are set to zero.
	In Figure~\ref{peanut-result102} the resulting parameterization is depicted for when we set $\f q = -\f C_{B}^1$. This is a suitable choice, since the domain is star shaped with respect to the origin.
	The patch parameterizations in Figures~\ref{peanut-result11} and \ref{peanut-result12} are obtained for $ \alpha = 10^{-7}$  and  $ \beta = 10^{-7} $, respectively. As one can see, having $\beta > 0$ yields a smoother parameterization for $\f q = \f n_B^1$.
	In this case, the choice $\f q = -\f C_{B}^1$ as in Figure~\ref{peanut-result102} yields the largest patch, which can be most easily covered by a multi-cell domain.
\end{example}
 \begin{figure}[ht]\centering 
	\begin{subfigure}{.45\textwidth}
		\centering
		\includegraphics[width=0.5\textwidth]{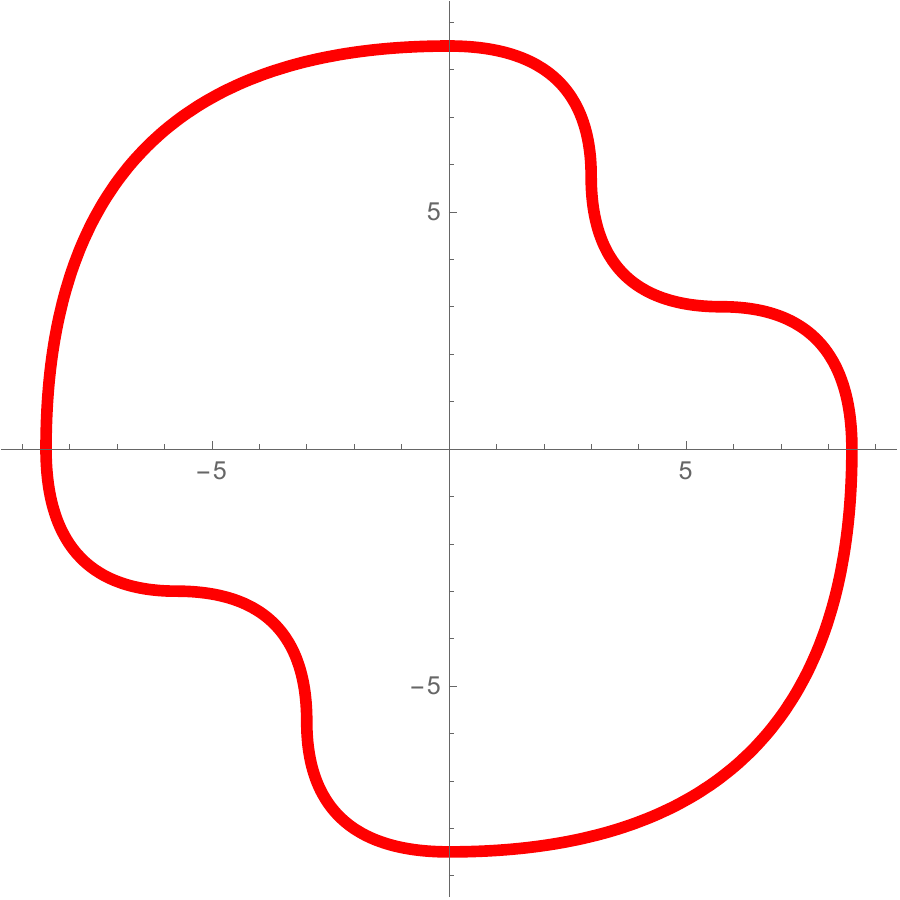}
		\subcaption{Peanut-shaped curve $\f C_B^1$.}\label{peanut-result00}
	\end{subfigure}\\
	\begin{subfigure}{.45\textwidth}
		\centering
		\includegraphics[width=0.5\textwidth]{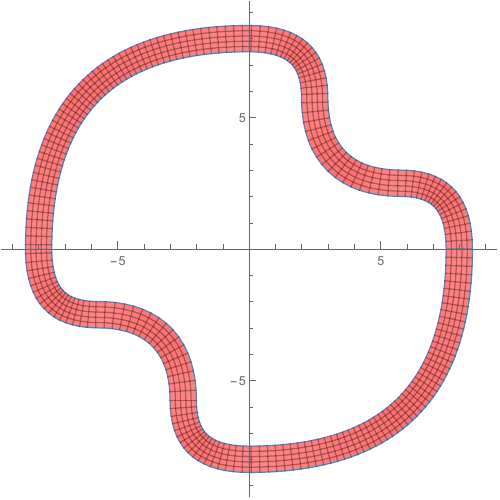}
		\subcaption{$d = 1$, $c = \frac{3}{4}$, $ \alpha = \beta = 0$, $\f q = \f n_B^1$.}\label{peanut-result101}
	\end{subfigure}%
	\begin{subfigure}{.45\textwidth}
		\centering
		\includegraphics[width=0.5\textwidth]{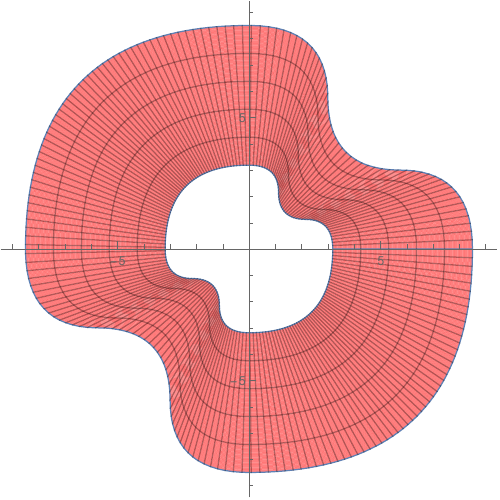}
		\subcaption{$d = 0.625 $, $c = \frac{3}{4} $, $ \alpha = \beta = 0$, $\f q = -\f C_B^1$.}\label{peanut-result102}
	\end{subfigure}\\
	\begin{subfigure}{.45\textwidth}
		\centering
		\includegraphics[width=0.5\textwidth]{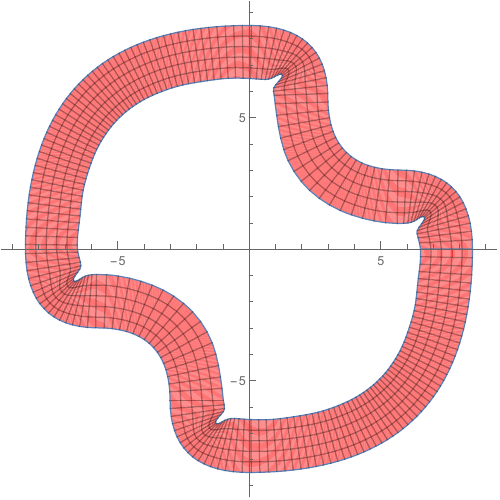}
		\subcaption{$d = 2 $, $c = \frac{3}{4}$, $ \alpha = 10^{-7}$, $\f q = \f n_B^1$.}\label{peanut-result11}
	\end{subfigure}%
	\begin{subfigure}{.45\textwidth}
		\centering
		\includegraphics[width=0.5\textwidth]{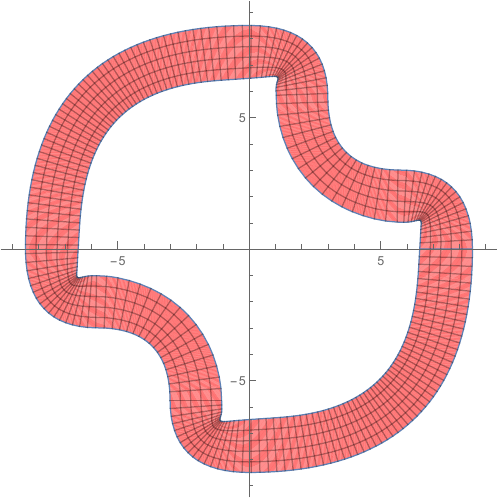}
		\subcaption{$d = 2 $, $c = \frac{3}{4}$, $ \beta= 10^{-7}$, $\f q = \f n_B^1$.}\label{peanut-result12}
	\end{subfigure}
	\caption{Different parameterizations of a domain inside a peanut-shaped curve~$\f C_B^1$. Note that in this figure (and in the following) the mesh is given for visualization purposes and does not correspond to the underlying B\'ezier mesh. Here~$ \f n_B^1 $ stands for the normal vector to the curve~$\f C_B^1$.}\label{peanut-result1}
\end{figure}

\begin{example}\label{example-with-normal-2}
	
	We consider a curve $\f C_B^2$ shaped like planet B-$612$, as shown in Figure~\ref{B612-curve00}.
	The curve is named in reference to the home of the little prince, from the book of the same name by Antoine de Saint-Exup\'ery.
	In Figure~\ref{B612-curve102} we visualize the resulting patch parameterization for $\f q = \f n_B^2$, $c=\frac{1}{2}$, $d=\frac{1}{2}$ and
	$\alpha = \beta = 0$.
	One can see that the parameterization is not valid, since it has self-intersections.
	Therefore, we set $c=\frac{1}{2}$, $d=\frac{1}{5}$ and $\alpha = \beta = 0$ in Figure~\ref{B612-curve201}  and $ \beta = 10^{-6}$ in Figure~\ref{B612-curve202}.
	The resulting parameterization in Figure~\ref{B612-curve202} does not have any self-intersections, but one part of the patch is very slim. In this case, it is not possible to cover the hole with a multi-cell domain, without having a very small mesh size.
	Hence, we need to define a different function $\f q$, such that the resulting parameterization becomes wider.
	Hence, instead of $\f q$ being the normal vector, we define $\f q = - \f C_B^2$. This is a suitable quasi-normal vector, since the curve is star-shaped with respect to a ball at the origin.
	The result is illustrated in Figure~\ref{B612-results}.
	In this case, covering the hole is easily possible.
\end{example}

\begin{figure}[ht]\centering
	\footnotesize 
	\begin{subfigure}{.45\textwidth}
		\centering
		\includegraphics[width=0.5\textwidth]{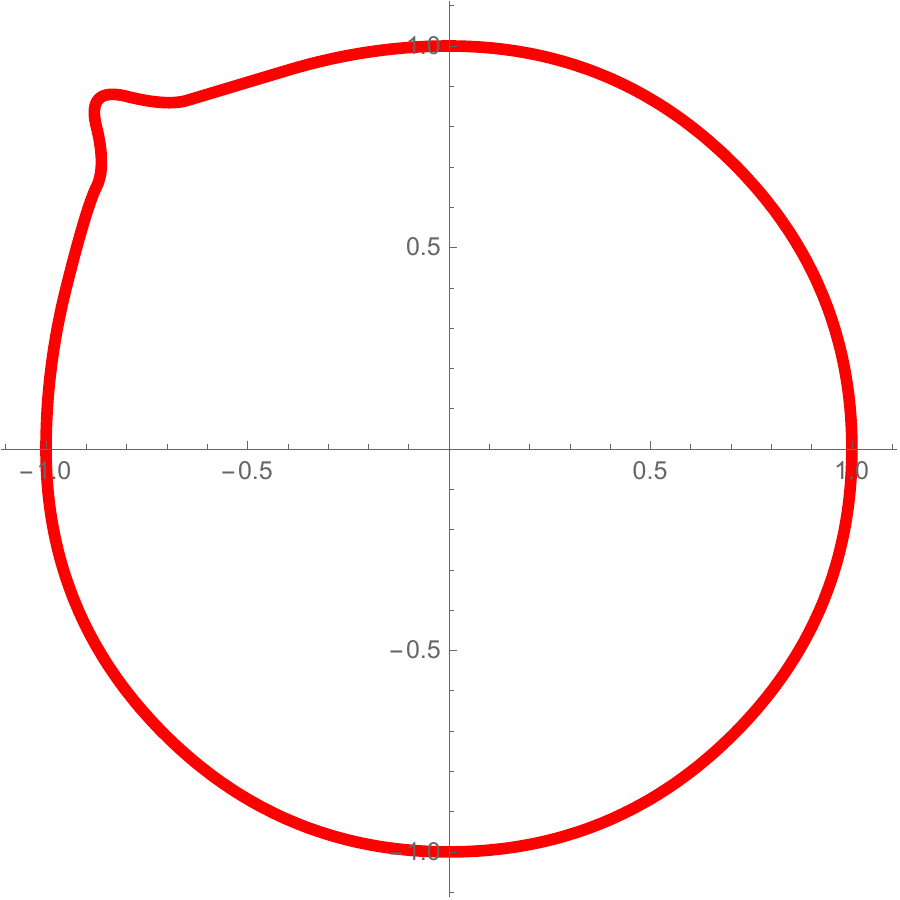}
		\subcaption{Planet B-$612$ curve $\f C_B^2$.}\label{B612-curve00}
	\end{subfigure}\\
	\begin{subfigure}{.45\textwidth}
		\centering
		\includegraphics[width=0.5\textwidth]{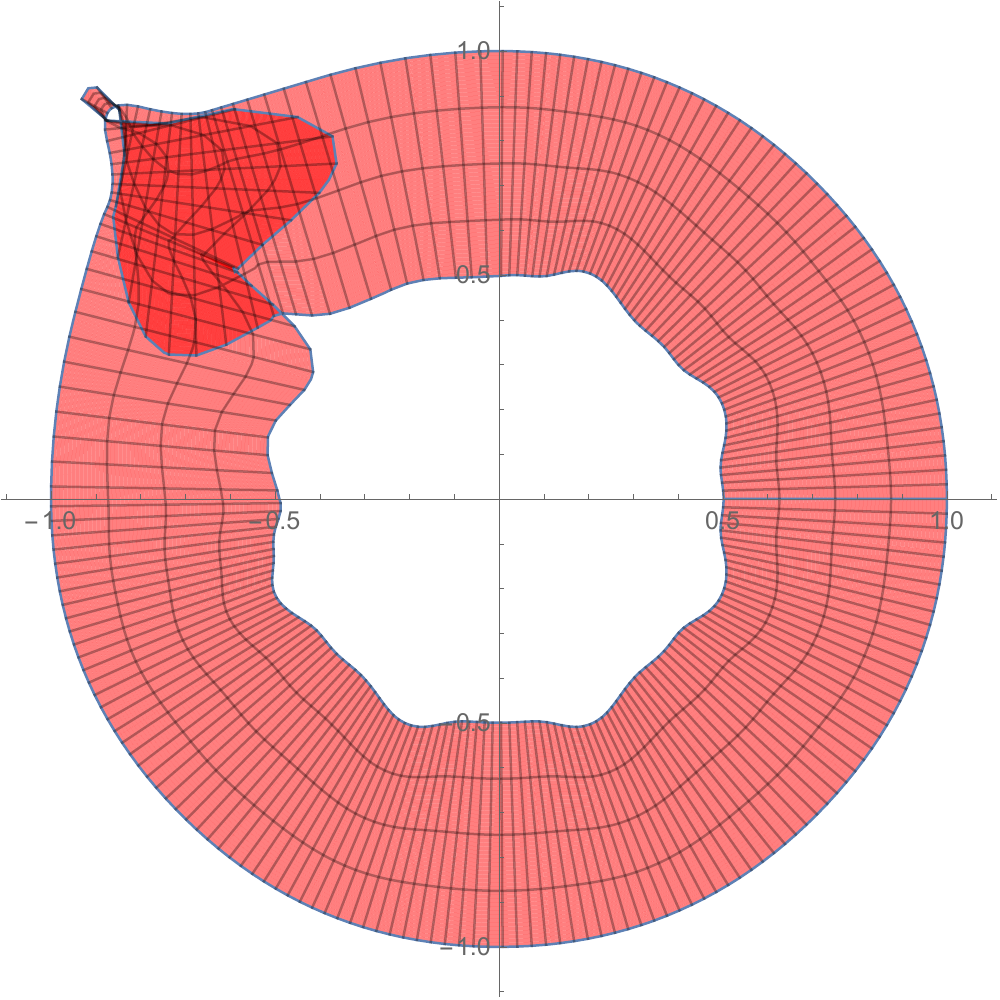}
		\subcaption{$d = \frac{1}{2} $, $c = \frac{1}{2}$, $ \alpha =\beta =0$, $\f q = \f n_B^2$.} \label{B612-curve102}  
	\end{subfigure}%
	\begin{subfigure}{.45\textwidth}
		\centering
		\includegraphics[width=0.5\textwidth]{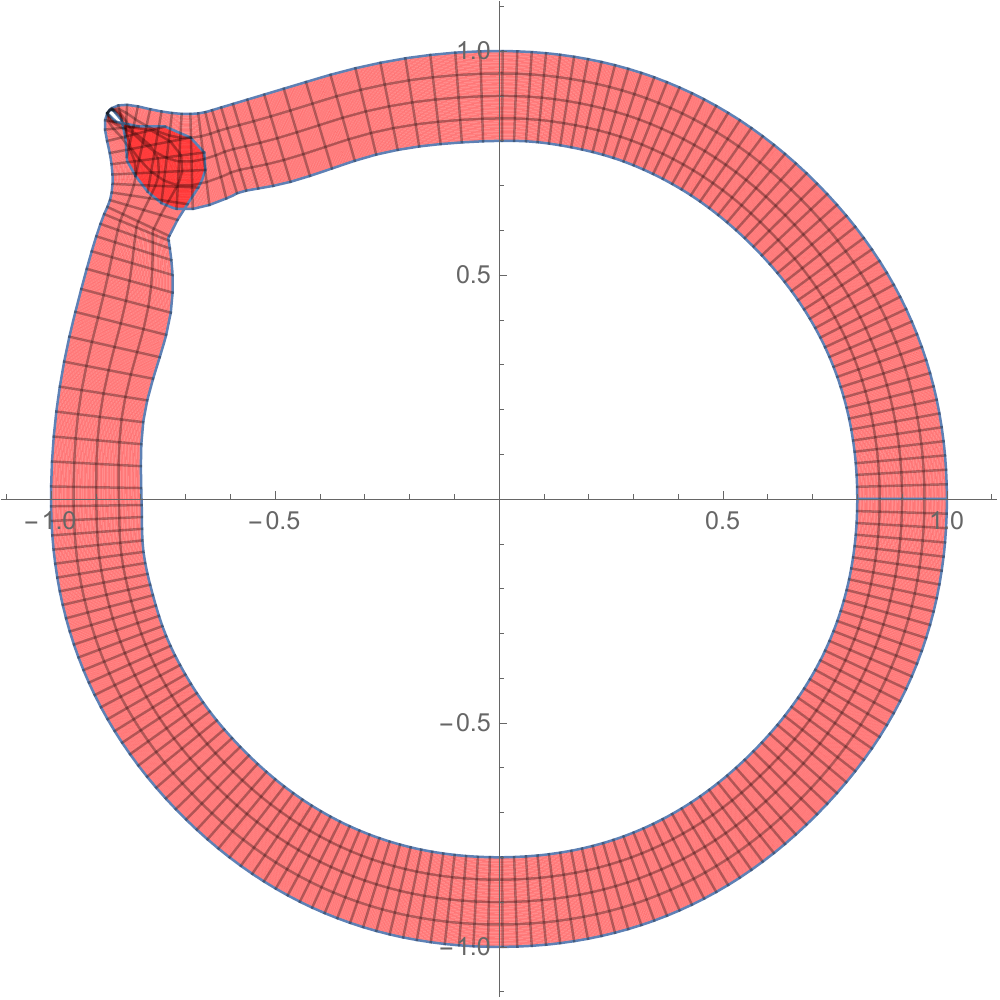}
		\subcaption{ $d = \frac{1}{5} $, $c = \frac{1}{2}$, $ \alpha =\beta =0$, $\f q = \f n_B^2$. }\label{B612-curve201}
	\end{subfigure}\\
	\begin{subfigure}{.45\textwidth}
		\centering
		\includegraphics[width=0.5\textwidth]{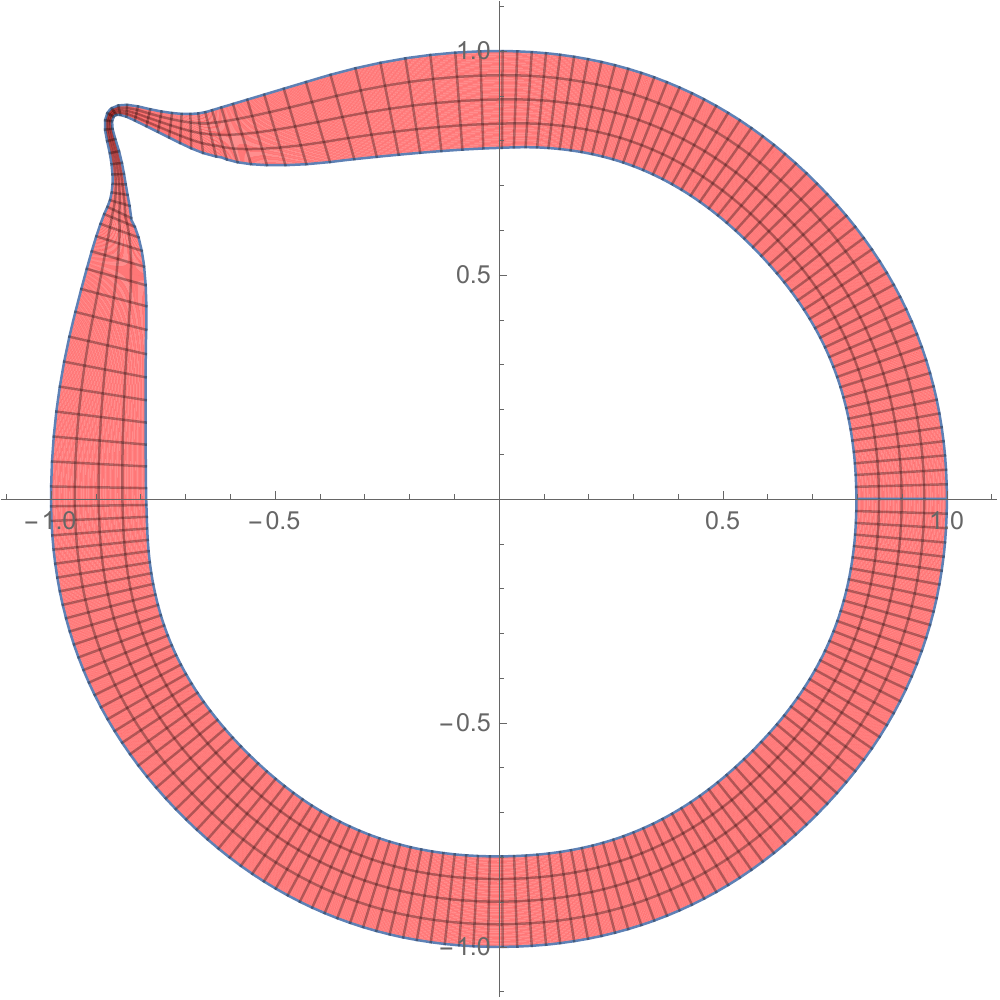}
		\subcaption{ $d = \frac{1}{5} $, $c = \frac{1}{2}$, $\beta = 10^{-6}$, $\f q = \f n_B^2$.}\label{B612-curve202} 
	\end{subfigure}%
	\begin{subfigure}{.45\textwidth}  
		\centering
		\includegraphics[width=0.5\textwidth]{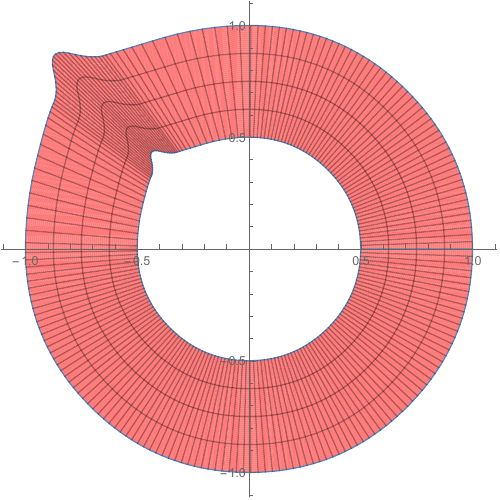}
		\subcaption{   $d = \frac{1}{2} $, $c = \frac{1}{2}$, $\alpha = \beta = 0$, $\f q= -\f C_{B}^2 $.} \label{B612-results}
	\end{subfigure}
	\caption{Different parameterizations of a domain inside the curve $\f C_B^2$ shaped like planet B-$612$.}\label{B612-curve000}
\end{figure}

\begin{example}\label{star-shape-curve}
	We consider the boundary curve $\f C_{B}^3$ of a star-shaped domain, as illustrated in Figure~\ref{star-comp00}.
	Here, we compare the parameterizations resulting from different choices of $\f q$ and different choices for the parameters $\alpha$ and $\beta$. The results using $\f q$ as the normal vector are illustrated in Figures~\ref{star-comp01} and~\ref{star-comp11}. The results using a more general quasi-normal vector (the normal vector of a circle $\f n_C $) are depicted in Figures~\ref{star-comp02} and~\ref{star-comp12}. As remarked earlier, the quasi-normal vector yields better results. Moreover, having $\beta > 0$ yields a smoother patch parameterization.
\end{example}
\begin{figure}[ht]\centering
	\footnotesize 
	\begin{subfigure}{.45\textwidth}  
		\centering
		\includegraphics[width=0.5\textwidth]{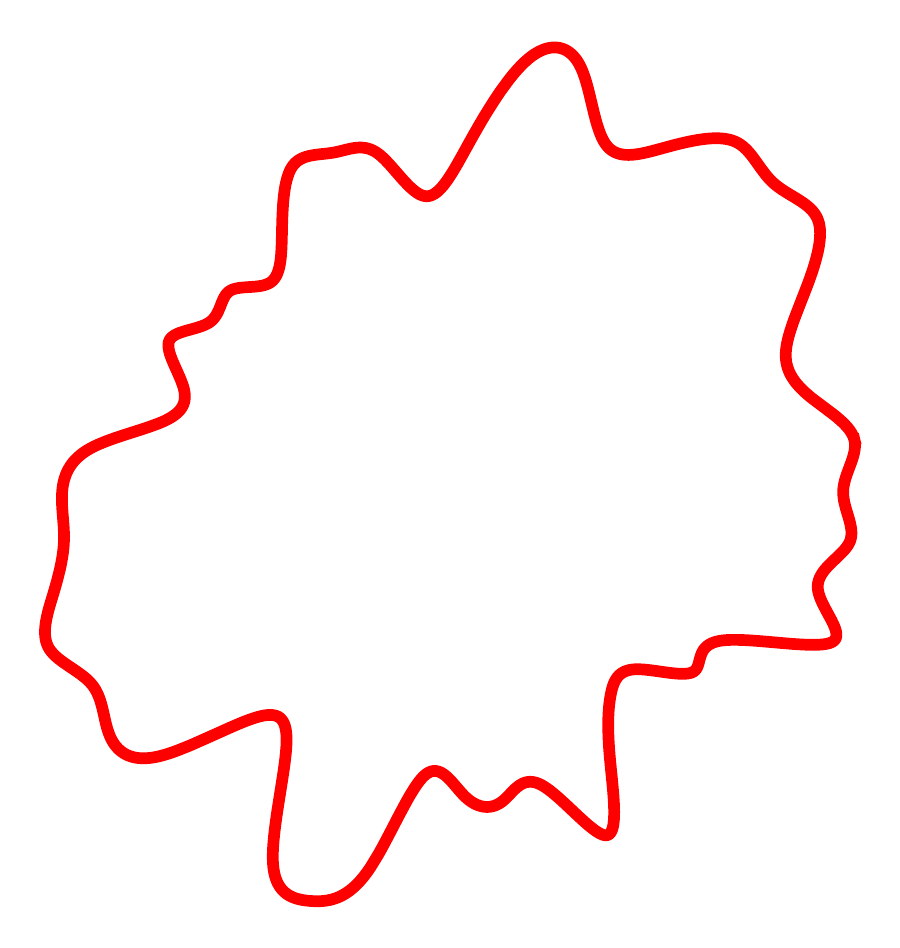}
		\subcaption{ Star-shaped curve $\f C_{B}^3$.} \label{star-comp00}
	\end{subfigure}\\
	\begin{subfigure}{.45\textwidth}  
		\centering
		\includegraphics[width=0.5\textwidth]{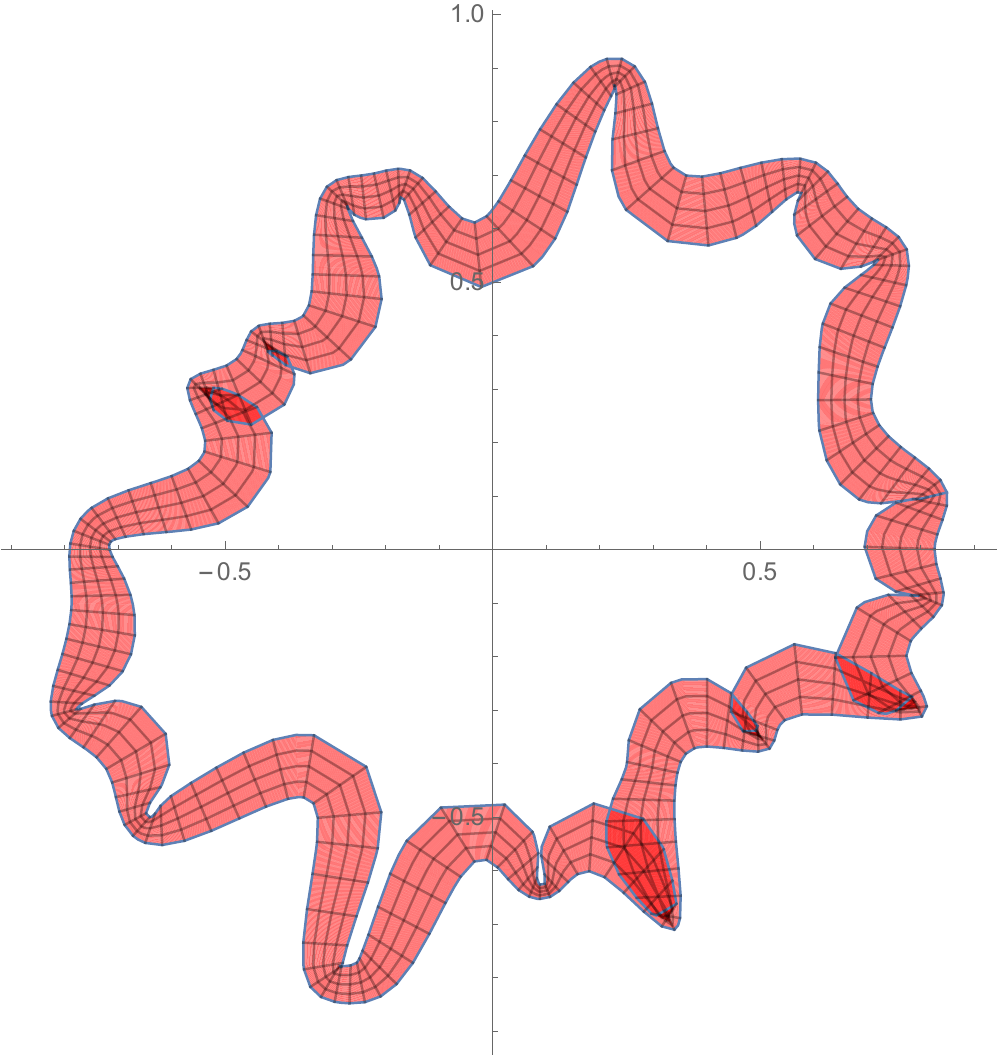}
		\subcaption{$d = \frac{1}{9} $, $c = \frac{1}{2}$, $ \f q = \f n_B^3 $, $\alpha = 3\times10^{-9}$.}\label{star-comp01}
	\end{subfigure}%
	\begin{subfigure}{.45\textwidth} 
		\centering 
		\includegraphics[width=0.5\textwidth]{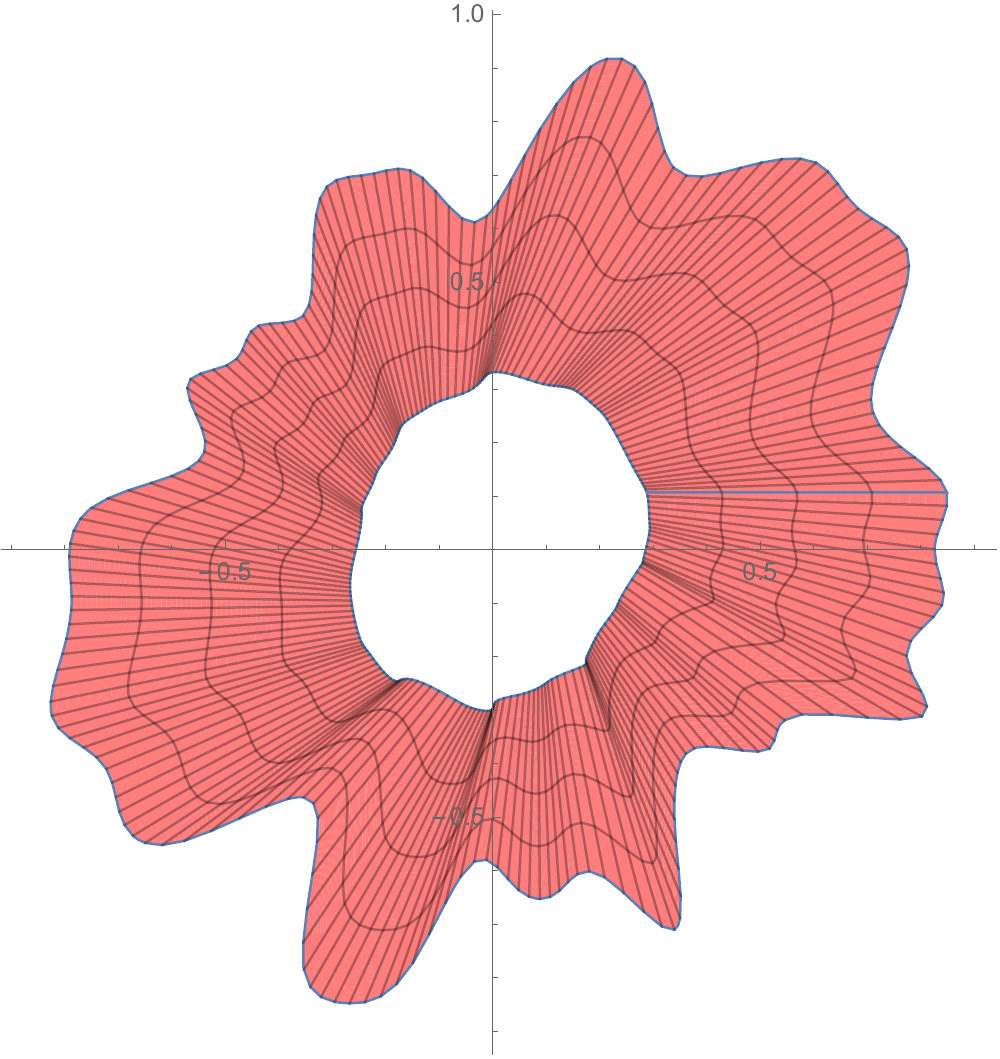}
		\subcaption{ $d = 1 $, $c = \frac{1}{2}$, $ \f q = \f n_C$, $\alpha = 9\times10^{-3}$.}\label{star-comp02}
	\end{subfigure}\\
	\begin{subfigure}{.45\textwidth}  
		\centering	\includegraphics[width=0.5\textwidth]{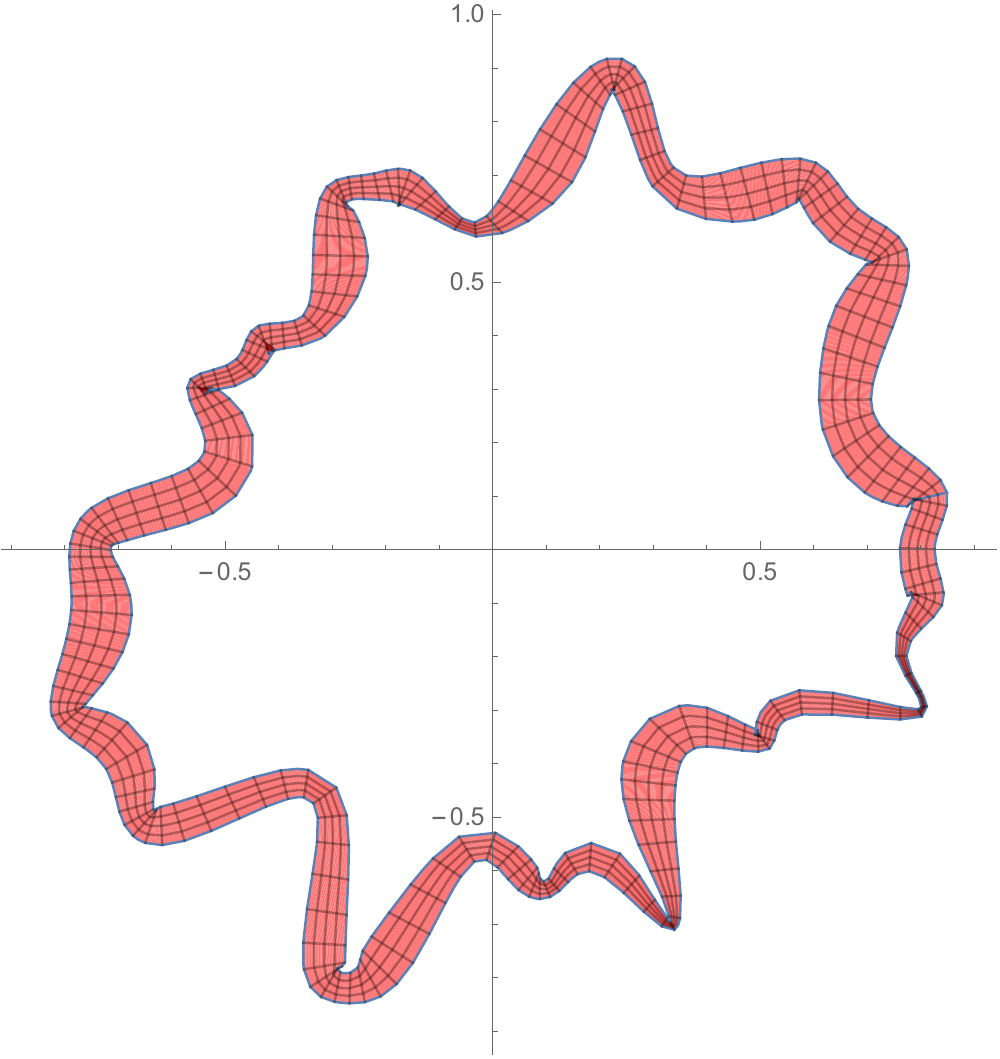}
		\subcaption{ $d = \frac{1}{9} $, $c = \frac{1}{2}$, $ \f q = \f n_B^3 $,  $\beta= 3\times 10^{-9}$.}\label{star-comp11}
	\end{subfigure}%
	\begin{subfigure}{.45\textwidth}  
		\centering
		\includegraphics[width=0.5\textwidth]{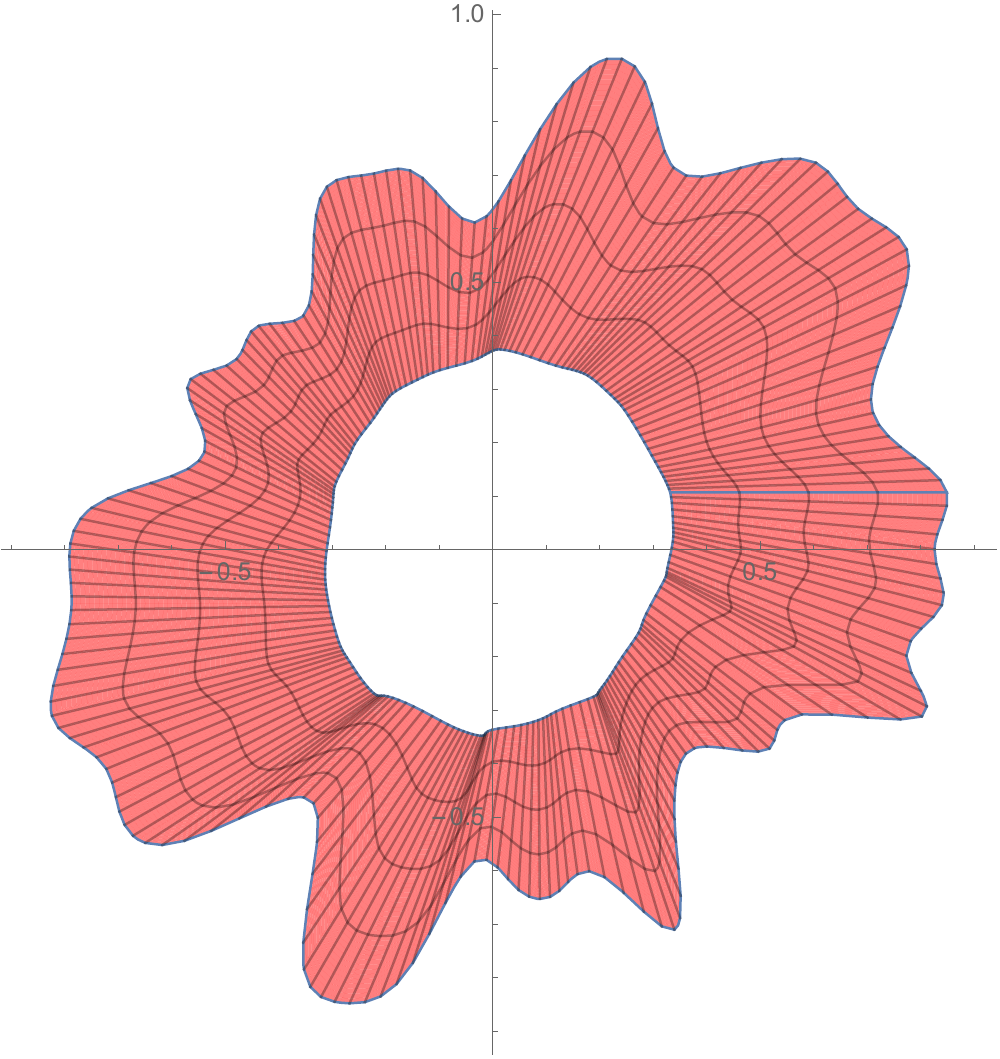}
		\subcaption{$d = 1$, $c = \frac{1}{2}$, $ \f q = \f n_C $, $\beta = 9\times 10^{-5}$. }\label{star-comp12}
	\end{subfigure}
	\caption{Applying the generalized offsetting algorithm on a star-shaped boundary curve.}
\end{figure}

\subsection{A discussion on constructions for the quasi-normal vector $\f q$}\label{comp-q(t)}

The goal of this section is to obtain a smooth quasi-normal vector for a boundary curve with fine details, as illustrated in Figure~\ref{new-ex01}. Such a quasi-normal vector can be computed as the normal vector of a smoother curve as it is shown in blue in Figure~\ref{new-ex02}. To this end we go through the following steps, given a boundary curve $\f C_B$:
\begin{itemize}
	\item We approximate $\f C_B$ by a smoother curve $\f C_S$, see Figures~\ref{new-ex01} and \ref{new-ex02}.
	\item We compute the normal vectors $\f n_B$ and $\f n_S$ to the curves $\f C_B$ and $\f C_S$, respectively.
	\item If, for each $t$, the normal vector $\f n_S(t)$ is a valid quasi-normal direction at $\f C_B(t)$, we set $\f q(t) = \f n_S (t)$.
	\item If this is not the case, we repeat the process with another approximation $\f C_S^*$ with $\|\f n_S^* - \f n_B \| < \|\f n_S - \f n_B \|$.
\end{itemize}

For a better understanding we present the following example.
\begin{example}\label{remark-example}
	We consider the boundary curve $\f C_B^4$ shown in Figure~\ref{new-ex01}.
	Moreover, in Figure~\ref{new-ex02}, the boundary curve $\f C_B^4$ is depicted together with a regularized, smoother curve $\f C_S$. The corresponding normal vectors are denoted by $\f n_B^4$ and $\f n_S$, respectively. The curve $\f C_S$ is an approximation of $\f C_B^4$, having fewer geometric details.
	In Figures~\ref{new-ex1} and  \ref{new-ex2} we depict the parameterizations where $\f q$ is chosen to be $\f q= \f n_B^4$ and $\f q= \f n_S$, respectively.
	In the latter case, the resulting patch parameterization is valid.
	Moreover, the hole in the middle can be covered easily using a multi-cell domain.
\end{example}

\begin{figure}[ht]\centering
	\footnotesize 
	\begin{subfigure}{.45\textwidth} 
		\centering
		\includegraphics[width=0.45\textwidth]{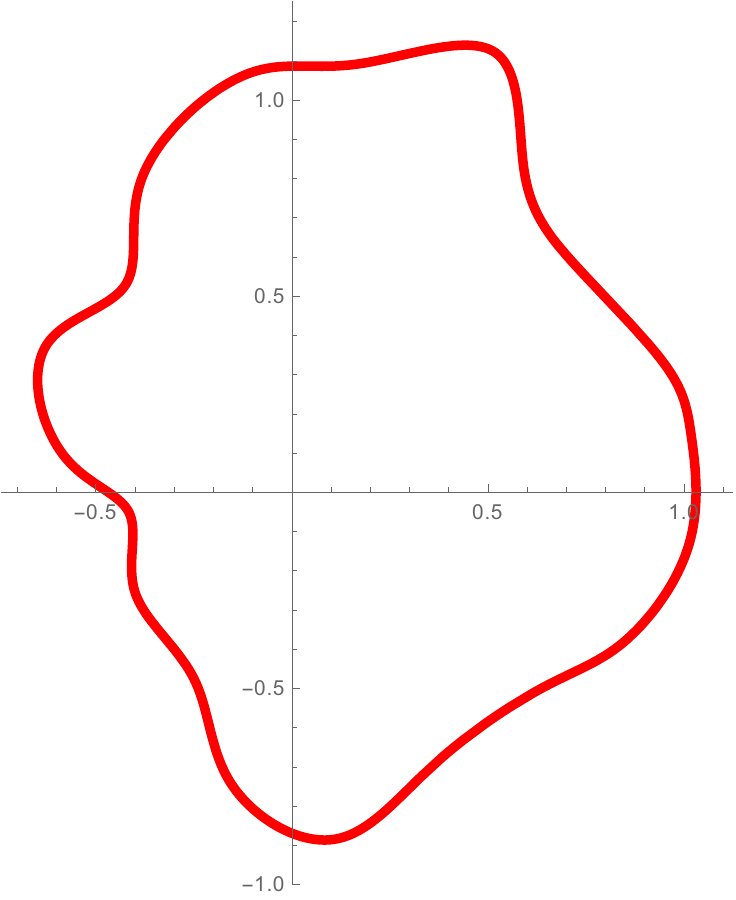}
		\subcaption{The given boundary curve $\f C_B^4$.}\label{new-ex01}
	\end{subfigure}%
	\begin{subfigure}{.45\textwidth} 
		\centering
		\includegraphics[width=0.45\textwidth]{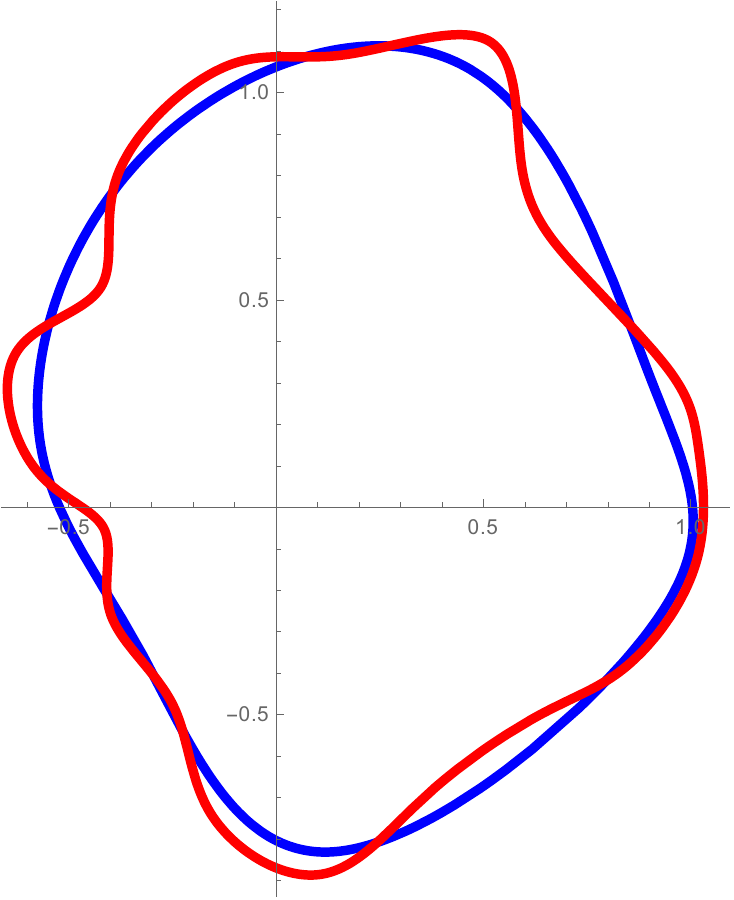}
		\subcaption{$\f C_B^4$ plotted together with $\f C_S$.}\label{new-ex02}
	\end{subfigure}\\
	\begin{subfigure}{.45\textwidth} 
		\centering
		\includegraphics[width=0.5\textwidth]{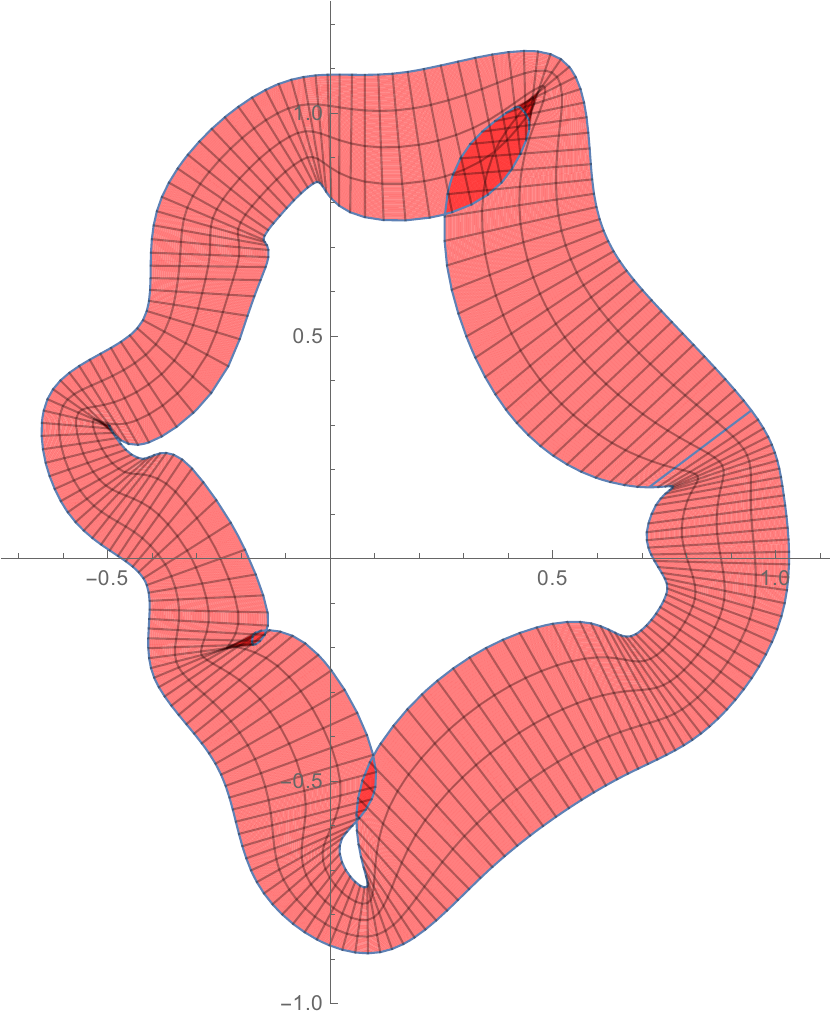}
		\subcaption{$d = \frac{2}{5}$, $c = \frac{1}{2}$, $ \f q = \f n_B^4 $, $\alpha = 10^{-7}$. }\label{new-ex1}		
	\end{subfigure}%
	\begin{subfigure}{.45\textwidth} 
		\centering
		\includegraphics[width=0.5\textwidth]{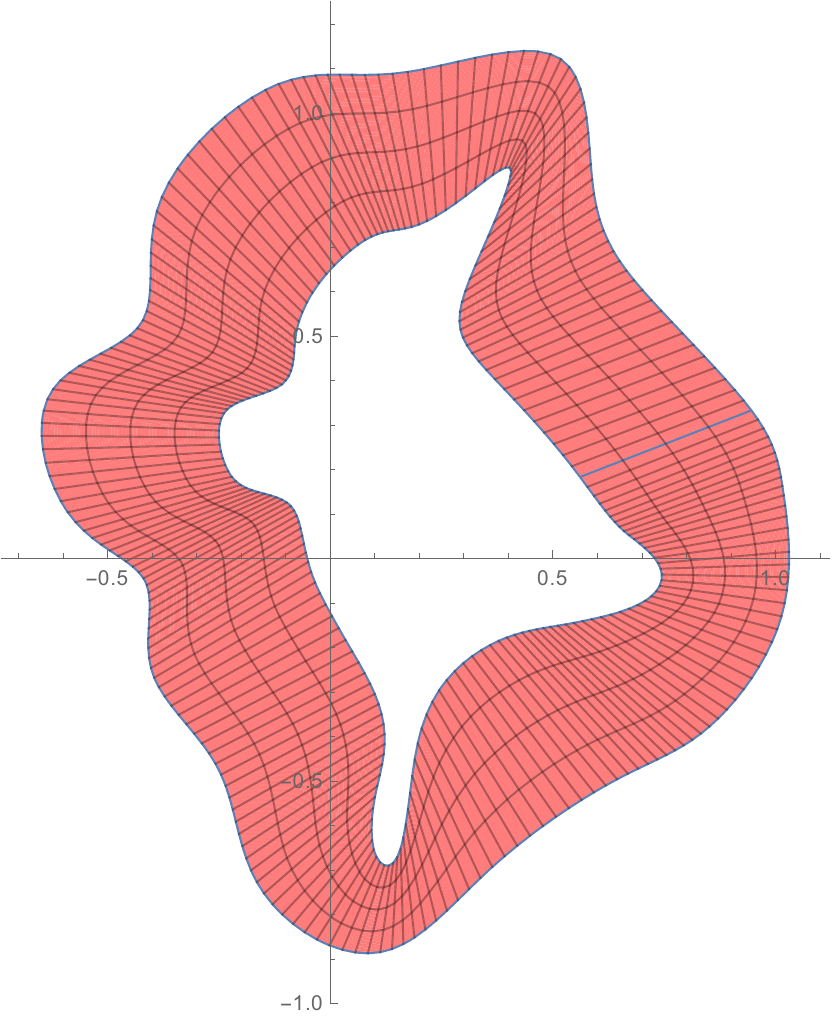}
		\subcaption{ $d = \frac{2}{5}$, $c = \frac{1}{2}$, $ \f q = \f n_S $, $\alpha = 10^{-7}$.}\label{new-ex2}
	\end{subfigure}
	\caption{Constructing a quasi-normal vector using a smooth approximation of a given boundary curve.}\label{new-ex}
\end{figure}

\section{Treatment of corners on the boundary}\label{sec:curves-with-corners}

In this section, we consider boundary curves with corners, both convex and non-convex. See Figure~\ref{b-curve} for a visualization. We explain how to extend the generalized offsetting algorithm to such curves. The algorithm yields a set of patches for each curve segment and each corner. The collection of patches can then be interpreted as a single ring-shaped spline manifold.

\setlength{\bw}{1cm}
\begin{figure}[ht]\centering
	\includegraphics[width=3\bw]{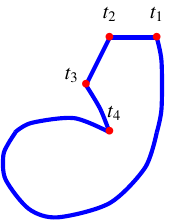}
	\caption{An example of a boundary curve $ \mathbf{C}_{B}$ with corners.}\label{b-curve}
\end{figure}

We assume to have given a $1$-periodic curve $\f C_{B} \in (\f S^p_h)^2$ with corners $t_i$, for $1\leq i\leq n$. Each corner corresponds to a knot of multiplicity $p$. The corners split the curve into segments $\f C_{B,i}$ with $t \in [t_{i-1}, t_i]$, i.e., 
\begin{equation*}\label{segments}
\begin{array}{l}
\f C_{B,1}(t) \colon t \in [t_1, t_2]\\
\f C_{B,2}(t) \colon t \in [t_2, t_3]\\
\vdots\\
\f C_{B,n}(t) \colon t \in [t_{n}, t_1 + 1].
\end{array}
\end{equation*}
We assume that no segment $\f C_{B,i}$ contains any additional corners, i.e., $\f C_{B,i} \in C^1([t_i, t_{i+1}])$ and $\|\f C_{B,i}'\|>0$. To obtain a ring-shaped parameterization for such a curve we need to go through the following steps, which are also visualized in Figure~\ref{special-case}:
\begin{itemize}
	\item For each segment $\f C_{B,i}$ we define a corresponding parameter interval $\f X_i^F$, with 
	\begin{equation*}
	\f X_i^F = \left\{
	\begin{array}{ll}
	\left[t_i+\delta,t_{i+1}-\delta \right] &\mbox{if both $t_i$ and $t_{i+1}$ are convex},\\
	\left[t_i+\delta,t_{i+1} \right] &\mbox{for $t_i$ convex, $t_{i+1}$ non-convex},\\
	\left[t_i,t_{i+1}- \delta \right] &\mbox{for $t_i$ non-convex, $t_{i+1}$ convex},\\
	\left[t_i,t_{i+1} \right] &\mbox{if both $t_i$ and $t_{i+1}$ are non-convex}.
	\end{array}\right.
	\end{equation*}
	We restrict the spline space $\f {S}^{p}_h$ to $\f X_i^F$, which is denoted by $\f{S}^{p}_h|_{\f X_i^F}$. We have given a quasi-normal vector $\f q_i$ for each segment $\f C_{B,i}$. The quasi-normals are not continuous at the corners, i.e., in general $\f q_{i-1}(t_i)\neq \f q_i(t_i)$.
	\item If $t_i$ is a convex corner, we construct a patch parameterization $\f P_i$ covering a neighborhood of the corner as a Coons patch interpolating the boundary curves 
	\begin{equation*}\label{cond-coons}
	\begin{array}{ll}
	\f P_i(u,0) &= \f C_{B,i-1} (t_i-u \cdot \delta )\\
	\f P_i(u,1) &= \f C_{B,i} (t_i+ \delta ) + u \cdot \mu_i(t_i + \delta) \, \f q_{i} (t_i + \delta)\\
	\f P_i(0,v) &= \f C_{B,i} (t_i+v \cdot \delta )\\
	\f P_i(1,v) &= \f C_{B,i-1} (t_i- \delta ) + v \cdot \mu_{i-1}(t_i - \delta)\, \f q_{i-1} (t_i - \delta).
	\end{array}
	\end{equation*}
	\item If $t_{i}$ is non-convex, we construct a patch parameterization $\f P_{i}$ as a parallelogram from given quasi-normal vectors $\f q_{i}(t_{i})$ and $\f q_{i -1}(t_{i})$.
	\item We apply the generalized offsetting algorithm to the segment $\f C_{B,i}$ on the parameter interval $\f X_i^F$ and obtain a parameterization $\tilde{\f{F}}_i:[0,1] \times \f X_i^F \rightarrow \mathbb{R}^2$. Instead of periodicity constraints, we now have to satisfy constraints to enforce continuity between the parameterizations $ \tilde{\f{F}}_i $ and the corner patches.
	\item For $  s \in [0,1] \mbox{ and } t \in \f X_i^F $, we define a local parameterization 
	\begin{equation*}
	\f{F}_i (s,t) = \f{C}_{B,i}(t) \cdot (1-s) + \f{C}_{I,i}(t) \cdot s,
	\end{equation*}
	such that the inner curve satisfies $\f{C}_{I,i} \in (\f{S}^{p}_h|_{\f X_i^F})^2$.
\end{itemize}
We visualize the construction in Figure~\ref{special-case}, where the corners $t_i$ and $t_{i+1}$ are convex, whereas $t_{i-1}$ is non-convex. Hence, we have $\f X_{i-1}^F = [t_{i-1},t_i-\delta]$ and $\f X_{i}^F = [t_{i}+\delta,t_{i+1}-\delta]$.

\begin{figure}[ht]\centering
	\includegraphics[width=6.5\bw]{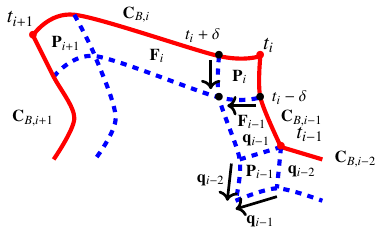}\\
	\includegraphics[width=5.5\bw]{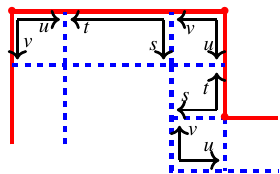}	  
	\caption{A collection of patches (top) and corresponding parameter domains (bottom).}\label{special-case}
\end{figure}

	\begin{figure}[ht]\centering
	\begin{subfigure}{0.9\textwidth} 
		\centering
			\includegraphics[width=4.5\bw]{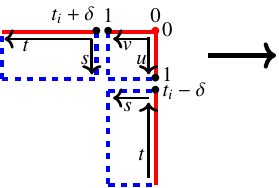}
			\hspace{0.2cm}
		\includegraphics[width=3.4\bw]{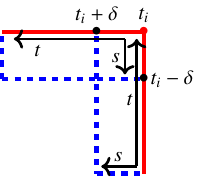}
		\subcaption{ Convex corner.}\label{conv-visualization}
	\end{subfigure}\\
	\begin{subfigure}{0.8\textwidth} 
		\centering
					\includegraphics[width=4.2\bw]{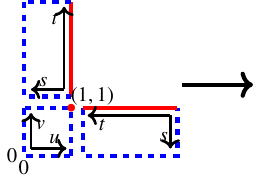}
		\hspace{0.2cm}
		\includegraphics[width=2.7\bw]{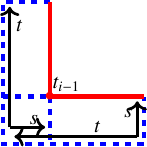}
		\subcaption{Non-convex corner.}\label{nonconv-visualization}
	\end{subfigure}
	\caption{Gluing the local patch parameterizations.}\label{parameter-domain-corner}
\end{figure}
\begin{remark}
	Note that the inner curves need to satisfy
	\begin{equation*}
	\f C_{I,i} (t_i + \delta) = \f C_{I,i-1} (t_i - \delta)
	\end{equation*}
	for each convex corner $t_i$. For each non-convex corner $t_{i-1}$ the angle between $\f q_{i-2}(t_{i-1})$ and $\f q_{i-1}(t_{i-1})$ must be greater than zero.
	Moreover, the value of $\delta$ can be chosen differently for each corner and each direction. Here, for simplicity, we assumed that $\delta$ is the same for all cases.
\end{remark}
As output we obtain a ring-shaped patch as shown in Figure~\ref{special-case} (top). This domain can be computed by creating a sequence of patches that are glued appropriately using a manifold-like structure. The continuity of the corresponding parameter domains is visualized in Figure~\ref{special-case} (bottom). This process is visualized for convex corners in Figure~\ref{conv-visualization} and for non-convex corners in Figure~\ref{nonconv-visualization}. The parameterizations presented in Examples~\ref{heart} and~\ref{drop} (see Figures~\ref{Domain-heart} and~\ref{Domain-drop}) are obtained using the above treatment for non-convex and convex corners, respectively. In both aforementioned examples, we apply the generalized offsetting algorithm to the segments without corners and cover the regions near the corners with appropriate Coons patches.

\section{Overlapping multi-patch formulation and isogeometric discretization}\label{sec:omp}

In this section we assume, for simplicity, that the domain $\Omega$ is constructed as the union of a periodic ring-shaped patch $\Omega^{\cR}$ and a rectangular patch $\Omega^{\cC}$, i.e., 
\[
\Omega = \Omega^{\cR} \cup \Omega^{\cC}.
\]
As a model problem, we consider a Poisson problem on the domain $\Omega$, with zero Dirichlet boundary conditions, which can be represented by a variational formulation as follows.
\begin{problem}\label{weak-poisson}
	Find $u \in H^1_0(\Omega)$, such that
	\begin{equation*}
	a(u,v) = \ell (v) \quad \forall v \in H_0^1(\Omega),
	\end{equation*}
	where
	\begin{equation*}\label{bilinear-poisson}
	a(u ,v) = \int_{\Omega} \nabla u \, \nabla v \,\mathrm d\boldsymbol{\xi}
	\quad
	\mbox{and}
	\quad
	\ell(v) = \int_{\Omega} f v \,\mathrm d\boldsymbol{\xi}.
	\end{equation*}
\end{problem}
We adopt the notation and definitions from~\cite{KARGARAN}. For $k \in \{\cR,\cC\}$ we define the following extension operator 
\begin{equation*}\label{assumptions2}
M^k v =
\begin{cases}
v  \quad \mbox{ on } \Gamma_C^k\\
0 \quad \mbox{ on } \Gamma_D^k, 
\end{cases}
\end{equation*} 
where, $\Gamma_C^k$ and $\Gamma_D^k$ are called the coupling and Dirichlet boundaries, respectively.

For $(u^\cR,u^\cC),(v^\cR,v^\cC) \in H^1(\Omega^\cR) \times H^1(\Omega^\cC)$, we define
\begin{equation*}
\mathcal{A}((u^\cR,u^\cC),(v^\cR,v^\cC)) = a^\cR(u^\cR,v^\cR)+a^{\cC}(u^\cC,v^\cC)
\end{equation*}
and
\begin{equation*}
\mathcal{L}(v^\cR,v^\cC) = \ell^\cR(v^\cR)+\ell^{\cC}(v^\cC),
\end{equation*}
where 
\begin{equation*}
a^k(u ,v) = \int_{\Omega^k} \nabla u \, \nabla v \,\mathrm d\boldsymbol{\xi}
\end{equation*}
and
\begin{equation*}
\ell^k(v) = \int_{\Omega^k} f v \,\mathrm d\boldsymbol{\xi}, \quad k \in \{ \cR , \cC \}.
\end{equation*}
Moreover, we introduce functions $u^k_0 \in H^1_0(\Omega^k)$, satisfying
\begin{equation}\label{solution}
u^k= u^k_0 + M^k u^{k'} \quad k \in \{ \cR,\cC \},
\end{equation}
where $k'$ such that $\{ k,k' \} = \{ \cR,\cC \}$. 
In~\cite{KARGARAN} it was shown that equation~\eqref{solution} is solvable under mild conditions on $M^\cR$ and $M^\cC$. We obtain the following coupled problem.
\begin{problem}\label{var2}
	Find $(u^{\cR}_0,u^{\cC}_0) \in H_0^1(\Omega^{\cR}) \times H_0^1(\Omega^{\cC}) $ and $(u^{\cR}_M,u^{\cC}_M) \in H^1(\Omega^{\cR}) \times H^1(\Omega^{\cC}) $ such that
	\begin{eqnarray}\centering 
	\mathcal{A}((u_0^{\cR} + u^{\cR}_M,u_0^{\cC} + u^{\cC}_M),(v^{\cR}_0,v^{\cC}_0)) &=&  \mathcal{ L}(v^{\cR}_0,v^{\cC}_0),\label{variationaleq444} \\
	u^{\cR}_M - M^{\cR} (u^{\cC}_0 + u^{\cC}_M) &\equiv& 0 \label{variationaleq445}\\
	u^{\cC}_M - M^{\cC} (u^{\cR}_0 + u^{\cR}_M) &\equiv& 0, \label{variationaleq446}
	\end{eqnarray}
for all $ \forall  (v^{\cR}_0,v^{\cC}_0) \in H_0^1(\Omega^{\cR}) \times H_0^1(\Omega^{\cC})  $. 
\end{problem}
In Problem~\ref{var2}, we obtained the coupling constraints~\eqref{variationaleq445} and \eqref{variationaleq446}
by replacing $M^k u^{k'}$ with $u^k_M$ in~\eqref{solution}.

In the following we summarize the isogeometric discretization of Problem~\ref{var2}. Since $\Omega^{\cR}$ is a periodic ring-shaped patch and $\Omega^{\cC}$ is rectangular, we use standard periodic isogeometric functions for the discretization of the functions corresponding to $\Omega^{\cR}$ and standard B-spline basis functions for the functions corresponding to the rectangular patch $\Omega^{\cC}$. We assume that $\f G^{\cR}$ is a spline geometry mapping, which is periodic in the first direction, mapping the parameter domain $\widehat\Omega^{\cR} = \left[0,1\right[ \times \left]0,1\right[$ onto the physical subdomain $\Omega^{\cR}$, i.e., $\Omega^{\cR} = \f G^{\cR}(\widehat{\Omega}^{\cR})$. We assume that the rectangular patch $\Omega^{\cC}$ has a parameterization $\f G^{\cC} : \left]0,1\right[^2 \rightarrow \Omega^{\cC}$, with $\f G^{\cC} \in \mathbb{P}^{(1,1)}$. The isogeometric spaces can then be defined as
\begin{equation*}
\begin{split}
V_{h}^{\cR} &=\mbox{span} \left\{ \varphi\in L^2(\Omega^{\cR}): \varphi \circ \f G^{\cR} \in \mathcal{S}^p_{h,\mathrm{per}} \times \mathcal{S}^p_{h} \right\},\\
V_{h}^{\cC}&=\mbox{span} \left\{ \varphi\in L^2(\Omega^{\cR}): \varphi \circ \f G^{\cC} \in \mathcal{S}^p_{h} \times \mathcal{S}^p_{h} \right\},
\end{split}
\end{equation*}
where $\mathcal{S}^p_{h,\mathrm{per}}$ denotes a periodic spline space. We define the interior functions for each subdomain as 
\begin{equation*}
\begin{array}{ll}
V_{0h}^{\cR} & = V_{h}^{\cR} \cap H_0^1(\Omega^{\cR}),\\
V_{0h}^{\cC} & = V_{h}^{\cC} \cap H_0^1(\Omega^{\cC}).
\end{array}
\end{equation*}
The spaces of coupling functions are defined as follows
\begin{eqnarray*}
	V_{c h}^k &=& \mbox{span} \left\{ \beta_i^k \in V_h^ k \;|\; i\in \cI_c^k \right\} \subset H^1(\Omega^k), \mbox{with } \\
	\mathcal{I}^k_{c }  &=& \left\{ i \in \mathcal{I}^k \;|\; \mbox{supp}\beta_i^k \cap \Gamma_C^k \neq \emptyset \right\},
\end{eqnarray*}
where $k \in \{ \mathcal{R}, \mathcal{C}\}$. 
Hence, Problem~\ref{var2} can be discretized as follows.
\begin{problem}\label{discretevar2}
	Find $(u_{0h}^\mathcal{R},u_{0h}^\mathcal{C}) \in V_{0 h}^\mathcal{R} \times V_{0h}^\mathcal{C} $ and  $(u^1_{Mh},u^2_{Mh}) \in V_{c h}^\mathcal{R} \times V_{c h}^\mathcal{C} $ such that
	\begin{eqnarray*}
	\mathcal{A}(u_{0h}^\mathcal{R}  + u^\mathcal{R} _{Mh},u_{0h}^\mathcal{C} + u^\mathcal{C}_{Mh}),(v^\mathcal{R} _{0h},v^\mathcal{C}_{0h})) &=&   \mathcal{L}(v^\mathcal{R} _{0h},v^\mathcal{C}_{0h}) \label{variationaleq516} \\
	u^\mathcal{R} _{Mh} - M_h^\mathcal{R}  (u_{0h}^\mathcal{C} + u^ \mathcal{C}_{Mh}) &=& 0  \quad \mbox{ on }  \Omega^\mathcal{R}  \label{variationaleq517}\\
	u^\mathcal{C}_{Mh} - M_h^\mathcal{C} (u_{0h}^\mathcal{R}  + u^\mathcal{R} _{Mh}) &=& 0 \quad \mbox{ on }  \Omega^\mathcal{C} \label{variationaleq518}
	\end{eqnarray*}
for all $ (v^\mathcal{R} _{0h},v^\mathcal{C}_{0h}) \in V_{0 h}^\mathcal{R} \times V_{0h}^\mathcal{C} $.
\end{problem}
In Problem~\ref{discretevar2}, $M_h^\mathcal{R}$ and $M_h^\mathcal{C}$ are suitable discretizations of the operators $M^\mathcal{R}$ and $M^\mathcal{C}$. 
We assume that the discretized extension operators are \emph{collocation-based extension operators} (CEO) as defined in~\cite[Section 4.4]{KARGARAN}.

\section{Numerical experiments}\label{Numerical-examples}

In all numerical experiments we solve the Poisson problem using the OMP method. Except for Example~\ref{peanut}, we always consider the exact solution
\begin{equation*}\label{exact-solution}
u(x,y) = \sin(\pi  x) \sin(\pi y).
\end{equation*}
The domain parameterizations were created using the generalized offsetting algorithm presented in Section~\ref{sec-algorithm-1}. In Examples~\ref{peanut} and~\ref{new-star} we consider domains with a smooth boundary. In Examples~\ref{heart} and~\ref{drop} we consider domains with corners that are convex (inner angle $<\pi$) and non-convex (inner angle $>\pi$), respectively.

\begin{example}\label{peanut}
We consider the peanut shaped domain from Example~\ref{example-with-normal}. The overlapping patches are depicted in Figure~\ref{Domain-peanut}. The parameterization of the ring-shaped patch is the one depicted in Figure~\ref{peanut-result102}, parameterized by periodic quadratic B-splines. We solve for the exact solution
	\begin{equation*}
	u(x,y) = \sin( x) \sin( y).
	\end{equation*}
	In Figure~\ref{localfighHalfCirc_NumSol0-peanut} we plot the numerical solution for a total of 1080 DOFs and using quadratic B-splines. $L^2$ and $H^1$ errors are shown in Figure~\ref{fighHalfCirc_H^1err0-peanut}. 
	The observed convergence rates are optimal.
\end{example}

\begin{figure}[ht]\centering
	\footnotesize
	\begin{subfigure}{0.8\textwidth} 
		\centering
					\includegraphics[width=9\bw]{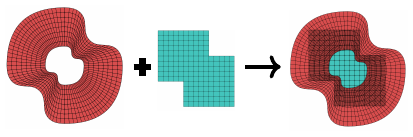}
		\subcaption{ Peanut shaped domain inside curve $\f C_B^1$.}\label{Domain-peanut}
	\end{subfigure}
	\\
	\begin{subfigure}{0.8\textwidth} 
		\centering
		\includegraphics[width=9\bw]{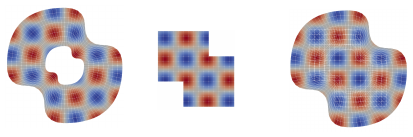}	
		\subcaption{Local solutions (left and center), both solutions plotted together (right).}\label{localfighHalfCirc_NumSol0-peanut}
	\end{subfigure}\\[6pt]
	\begin{subfigure}{0.9\textwidth} 
		\centering
			\includegraphics[width=5\bw]{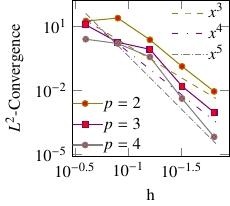}%
			\hspace{10pt}
				\includegraphics[width=5\bw]{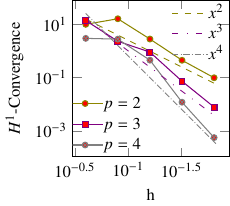}
		\subcaption{$L^2$  (left) and $H^1$ (right) errors for $p=2,3,4$.}\label{fighHalfCirc_H^1err0-peanut}
	\end{subfigure}
	\caption{Numerical results for solving the Poisson problem on a peanut shaped domain.}
\end{figure} 

\begin{example}\label{new-star}
	We consider the domain from Example~\ref{remark-example}. The parameterization of the ring-shaped patch is the one depicted in Figure~\ref{new-ex2}, parameterized by periodic cubic B-splines. The resulting OMP representation is depicted in Figure~\ref{Domain-shaped-boundary0}.
	
	We show the numerical solution for a total of 2263 DOFs and using cubic B-splines, in Figure~\ref{localfighHalfCirc_NumSol0-shaped-boundary0}. The $L^2$ and $H^1$ errors for discretizations of degree $p=3$ and $p=4$ are shown in Figure~\ref{fighHalfCirc_H^1err0-shaped-boundary0}. All convergence rates are optimal.
\end{example}

\begin{figure}[ht]\centering
	\footnotesize
	\begin{subfigure}{0.8\textwidth} 
		\centering
			\includegraphics[width=9\bw]{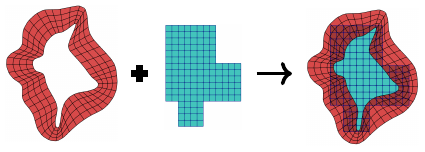}
		\subcaption{Domain inside curve $\f C_B^4$.}
		\label{Domain-shaped-boundary0}
	\end{subfigure}
	\\
	\begin{subfigure}{0.8\textwidth} 
		\centering\includegraphics[width=9\bw]{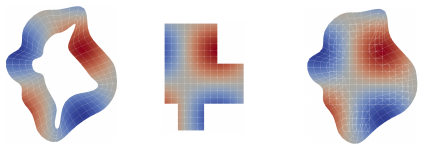}	
		\subcaption{Local solutions (left and center), both solutions plotted together (right).}\label{localfighHalfCirc_NumSol0-shaped-boundary0}
	\end{subfigure}
	\\[6pt]
	\begin{subfigure}{0.9\textwidth} 
		\centering
			\includegraphics[width=5\bw]{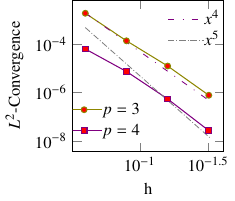}%
			\hspace{10pt}
				\includegraphics[width=5\bw]{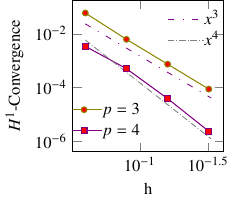}	
		\subcaption{$L^2$  (left) and $H^1$ (right) errors for $p = 3,4$.}\label{fighHalfCirc_H^1err0-shaped-boundary0}
	\end{subfigure}
	\caption{Numerical results for solving the Poisson problem on the domain with a random smooth boundary.}
\end{figure}
\begin{example}\label{heart}
	We consider a heart shaped domain, which is illustrated in Figure~\ref{Domain-heart}. In this example, the geometry is a non-convex domain with a corner near the top. The construction of the ring-shaped patch is done as explained in Section~\ref{sec:curves-with-corners}. The ring-shaped subdomain is parameterized with quadratic B-splines.
	
	The numerical solution for a total of 1460 DOFs and using quadratic B-splines is illustrated in Figure~\ref{localfighHalfCirc_NumSol0-heart}. $L^2$ and $H^1$ errors are shown in Figure~\ref{fighHalfCirc_H^1err0-heart}. 
	The convergence rates of all errors are optimal.
\end{example}

\begin{figure}[ht]\centering
	\footnotesize
	\begin{subfigure}{0.8\textwidth} 
		\centering
			\includegraphics[width=9\bw]{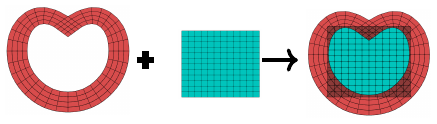}	
		\subcaption{Heart shaped domain.}\label{Domain-heart}
	\end{subfigure}
	\\
	\begin{subfigure}{0.8\textwidth} 
			\centering\includegraphics[width=9\bw]{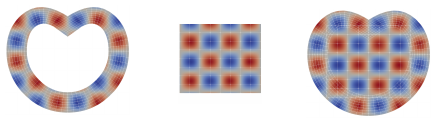}	
		\subcaption{Local solutions (left and center), both solutions plotted together (right).}\label{localfighHalfCirc_NumSol0-heart}
	\end{subfigure}
	\\[6pt]
	\begin{subfigure}{0.9\textwidth} 
		\centering
			\includegraphics[width=5.1\bw]{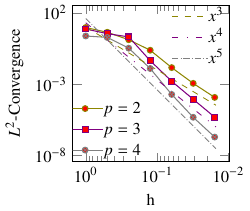}%
				\hspace{10pt}\includegraphics[width=5.1\bw]{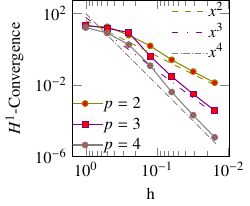}		
		\subcaption{$L^2$  (left) and $H^1$ (right) errors for $p=2,3,4$.}\label{fighHalfCirc_H^1err0-heart}
	\end{subfigure}
	\caption{Numerical results for solving the Poisson problem on a heart shaped domain.}
\end{figure} 
\begin{example}\label{drop}
	The drop shaped domain illustrated in Figure~\ref{Domain-drop} is composed of two overlapping patches. The geometry is convex with a corner at the top.
	The ring-shaped part is constructed with quadratic B-splines according to the strategy in Section~\ref{sec:curves-with-corners}.
	
	In Figure~\ref{localfighHalfCirc_NumSol0-drop} we show the numerical solution for a total of 1460 DOFs and using quadratic B-splines. $L^2$ and $H^1$ errors are shown in Figure~\ref{fighHalfCirc_H^1err0-drop}. 
	The observed convergence rates of all errors are optimal.
\end{example}

\begin{figure}[ht]\centering
	\footnotesize
	\begin{subfigure}{0.8\textwidth} 
		\centering
			\includegraphics[width=8.5\bw]{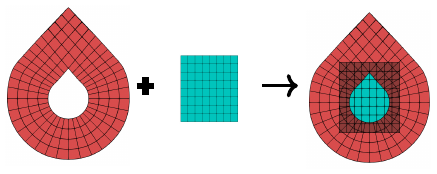}
		\subcaption{Drop shaped domain.}\label{Domain-drop}
	\end{subfigure}
	\\
	\begin{subfigure}{0.8\textwidth} 
		\centering
			\includegraphics[width=8.5\bw]{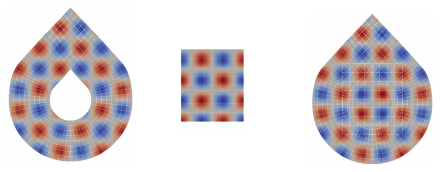}
		\subcaption{Local solutions (left and center), both solutions plotted together (right).}\label{localfighHalfCirc_NumSol0-drop}
	\end{subfigure}\\[6pt]
	\begin{subfigure}{0.9\textwidth} 
		\centering
			\includegraphics[width=5.1\bw]{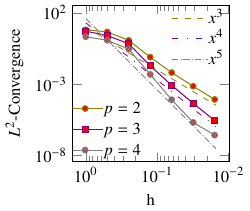}%
			\hspace{10pt}
				\includegraphics[width=5.1\bw]{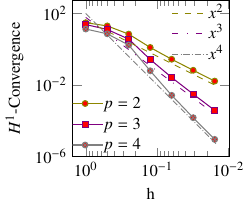}
		\subcaption{$L^2$ (left) and $H^1$ (right) errors for $p=2,3,4$.}\label{fighHalfCirc_H^1err0-drop} 
	\end{subfigure}
	\caption{Numerical results for solving the Poisson problem on a drop shaped domain.}
\end{figure}

\section{Conclusion}
In this paper, we proposed an offset-based overlapping domain parameterization (OODP) method for IGA.
In this method we generate an inner offset curve from a given regularly parameterized boundary curve by solving a regularized optimization problem. We penalize first and/or second derivatives to make the curve as smooth as possible. From this inner curve and the given outer curve we obtain a ring-shaped patch with a hole. This hole is then covered with a multi-cell domain.

We study the influence of the parameters in the proposed algorithm via several numerical examples. We show that using a quasi-normal vector, instead of the exact normal vector, for defining the inner offset curve in the algorithm generates a better parameterization for most of the considered cases. In addition, we apply the algorithm to boundary curves with convex and non-convex corners and we explain how the method can handle such cases. 
Finally, we use the OMP method as introduced in~\cite{KARGARAN} to solve PDEs on the proposed domain parameterizations. 

The OODP method can be applied to a wide range of given boundary curves in 2D. As we have demonstrated, it is always possible to cover complex domains with pairwise overlapping subdomains. However, the OODP method requires proper tuning of the involved parameters and a suitable choice for the quasi-normal vector. Hence, it is of vital interest to study the dependence of the method on its parameters and to devise a suitable parameter adjustment strategy.

In many cases, the parameterization strategy becomes more straightforward if we allow to have multiple overlaps (where three or more subdomains are overlapping). Therefore, for solving a PDE on such parameterizations, the OMP method, which is defined for pairwise overlaps, needs to be extended to configurations having multiple overlaps. We intend to extend the method in such a way in the future.

The OMP method proposed in \cite{KARGARAN} is developed for the discretization of second order PDEs. Solving higher order PDEs requires higher order coupling between the patches, e.g. for fourth order problems $C^1$ coupling is needed. Over matching multi-patch domains this leads to additional constraints on the multi-patch parameterization, which require a reparameterization as developed in~\cite{kapl}. An extension of the OMP method to fourth order problems might circumvent this issue.

Even though we only consider B-spline parameterizations in the numerical examples, the proposed method is also applicable for NURBS boundary curves. Moreover, the OODP method can be generalized also to some 3D domains. We want to highlight here the possibility to parameterize 2.5D domains, which are obtained from planar or surface domains by extrusion or sweeping. Moreover, objects that are constructed as Boolean unions of 2.5D domains can also be parameterized using overlapping patches. Such extensions are relatively straightforward and possess numerous applications, as many domains of practical relevance can be constructed in such a way. An extension to the full 3D case is however more complicated, as there are many possible configurations of boundary surfaces, which are in general given as trimmed B-spline or NURBS patches. Therefore, the parameterization of full 3D domains needs further research.
\section{Acknowledgments}

The research was supported by the strategic program ``Innovatives O\"O 2010
plus'' by the Upper Austrian Government, by the FWF 
together with the Upper Austrian Government through the
project P~30926-NBL, by Linz Institute of Technology and the government of Upper Austria through the project LIT-2019-8-SEE-116, by the Austrian Ministry for Transport, Innovation and Technology (BMVIT), the Federal Ministry for Digital and Economic Affairs (BMDW), and the Province of Upper Austria in the frame
of the COMET-Competence Centers for Excellent Technologies Program managed by Austrian Research Promotion Agency FFG, the COMET Module S3AI and by the "Austrian COMET-Programme" (Project InTribology, no. 872176).

\appendix
\section{Proof of Theorem~\ref{mu-max-theorem}}\label{prooftheorem1}

\begin{proof}
The parameterization~$\tilde{\f{F}}$ is given as in~\eqref{parameterization}, where 
\begin{equation*}
\f{C}_{B} = ( C^1(t) , C^2(t))\quad \mbox{ and } \quad \f{q}(t) = ( q^1(t) , q^2(t)).
\end{equation*}
We have that~$\tilde{\f{F}}$ is regular, if the Jacobian determinant of the mapping is always negative. We obtain
\begin{equation*}
\frac{\partial \tilde{\f{F}}(s,t)}{\partial s} = \mu(t) \f q(t)
\end{equation*}
and
\begin{equation*}
\frac{\partial \tilde{\f{F}}(s,t)}{\partial t} = \f{C}_B'(t) + s (\mu'(t) \f q(t) + \mu(t) \f q'(t)),   
\end{equation*}
Therefore, the Jacobian matrix of $\tilde{\f{F}}(s,t) $ can be written as follows
\begin{equation*}
J = \left(
\begin{array}{cc}
\mu(t) \f q(t) & \f{C}_B'(t) + s (\mu'(t) \f{q}(t) + \mu(t) \f{q}'(t))
\end{array}
\right).
\end{equation*}
We denote the determinant of $J$ by $\det(J) = D$. 
In the following we obtain 
\begin{eqnarray*}
D &=& \mu(t) \det(\f q(t) ,  \f{C}_B'(t))+ (\mu(t)  s \mu'(t) )\underbrace{\det(\f{q}(t) ,\f{q}(t) ) }_{= 0}\\
&+& s \mu^2(t) \det(\f{q}(t) ,\f{q'}(t)) 
\end{eqnarray*}
which should be negative for all $s\in[0,1]^2$ and for all $t\in\mathbb{R}$, i.e., 
\begin{equation}\label{start-upper-bound}
\mu(t) \det(\f q(t) ,  \f{C}_B'(t)) + s\mu^2(t) \det(\f{q}(t) ,\f{q}'(t) ) < 0.
\end{equation}
Equation~\eqref{start-upper-bound} is linear with respect to $ s $. Hence, it suffices to satisfy the equation for $s=0$ and $s=1$. For $s = 0$ and we obtain
\begin{equation*}
\mu(t) \det(\f q(t) ,  \f{C}_B'(t)) < 0,
\end{equation*}
Since $\det(\f q(t) ,  \f{C}_B'(t)) $ is negative, we need $\mu(t) > 0$.
For $s = 1$ we get
\begin{equation*}
\mu(t) \det(\f q(t) ,  \f{C}_B'(t)) + \mu^2(t) \det(\f{q}(t) ,\f{q'}(t) ) < 0,
\end{equation*}
or
\begin{equation*}
\mu(t) \det(\f{q}(t) ,\f{q'}(t))  < - \det(\f q(t) ,  \f{C}_B'(t)) .
\end{equation*}
If 
\[
\det(\f{q}(t) ,\f{q'}(t)) > 0
\]
we obtain
\begin{equation*}
\mu(t)  <  \frac{\det(\f{C}_B'(t),\f q(t)) }{\det(\f{q}(t) ,\f{q'}(t))}.
\end{equation*}
On the other hand, if 
\[
\det(\f{q}(t) ,\f{q'}(t)) \leq 0
\]
we have
\begin{equation*}
\mu(t) \det(\f{q}(t) ,\f{q'}(t))  \leq 0 < -\det(\f q(t),\f{C}_B'(t))
\end{equation*}
which is satisfied for all $t$. This completes the proof.
\end{proof}

\end{document}